\documentclass[11pt,letterpaper]{amsart}
\usepackage{amsmath,amssymb,amsthm,amscd,graphicx,mathrsfs,url,dsfont,enumitem}
\usepackage[usenames,dvipsnames]{color}
\usepackage[colorlinks=true,linkcolor=Blue,citecolor=Green,urlcolor=Blue]{hyperref}
\usepackage{dsfont}
\usepackage[super]{nth}
\usepackage[open, openlevel=2, depth=3, atend]{bookmark}
\hypersetup{pdfstartview=XYZ}
\usepackage{epigraph}
\usepackage{mathtools}
\usepackage{enumitem}
\usepackage{mathrsfs}

\usepackage{graphicx} 
\usepackage{float} 
\usepackage{subfigure} 
\usepackage[usenames,dvipsnames]{color}
\usepackage{comment}

\makeatletter
\@namedef{subjclassname@2020}{\textup{2020} Mathematics Subject Classification}
\makeatother

\let\oldtocsection=\tocsection
\let\oldtocsubsection=\tocsubsection
\let\oldtocsubsubsection=\tocsubsubsection
\renewcommand{\tocsection}[2]{\hspace{0em}\oldtocsection{#1}{#2}}
\renewcommand{\tocsubsection}[2]{\hspace{1em}\oldtocsubsection{#1}{#2}}
\renewcommand{\tocsubsubsection}[2]{\hspace{2em}\oldtocsubsubsection{#1}{#2}}

\renewcommand{\tilde}{\widetilde}          
\DeclareMathSymbol{\leqslant}{\mathalpha}{AMSa}{"36} 
\DeclareMathSymbol{\geqslant}{\mathalpha}{AMSa}{"3E} 
\DeclareMathSymbol{\eset}{\mathalpha}{AMSb}{"3F}     
\renewcommand{\leq}{\;\leqslant\;}                   
\renewcommand{\geq}{\;\geqslant\;}                   
\renewcommand{\d}{\mathrm{d}}             
\renewcommand{\emptyset}{\eset}

\numberwithin{equation}{section}

\newtheorem{theorem}{Theorem}[section]
\newtheorem{lemma}[theorem]{Lemma}
\newtheorem{proposition}[theorem]{Proposition}
\newtheorem{corollary}[theorem]{Corollary}

\theoremstyle{remark}
\newtheorem{remark}{Remark}
\newtheorem{example}{Example}
\theoremstyle{definition}

\newcommand{\C}{\mathbb{C}}
\newcommand{\D}{\mathbb{D}}
\newcommand{\R}{\mathbb{R}}
\newcommand{\Z}{\mathbb{Z}}

\newcommand{\N}{\mathbb{N}}
\newcommand{\Q}{\mathbb{Q}}

\newcommand{\E}{\mathbb{E}}
\renewcommand{\P}{\mathbb{P}}
\renewcommand{\S}{\mathbb{S}}

\renewcommand{\Im}{\mathrm{Im}}
\renewcommand{\Re}{\mathrm{Re}}

\newcommand{\cT}{\mathcal{T}}

\newcommand{\cE}{\mathcal{E}}
\newcommand{\cC}{\mathcal{C}}
\newcommand{\cH}{\mathcal{H}}
\newcommand{\cF}{\mathcal{F}}

\newcommand{\cN}{\mathcal{N}}

\newcommand{\cV}{\mathcal{V}}

\newcommand{\cP}{\mathcal{P}}
\newcommand{\ind}{\mathds{1}}

\newcommand{\cW}{\mathcal{W}}
\newcommand{\cD}{\mathcal{D}}

\newcommand{\laweq}{\overset{\text{law}}{=}}
\newcommand{\cI}{\mathcal{I}}

\renewcommand{\hat}{\widehat}

\newcommand{\del}{\partial}

\newcommand{\bS}{\mathbf{S}}

\newcommand{\bv}{\mathbf{v}}

\newcommand{\norm}[1]{\left\Vert #1\right\Vert}

\newcommand{\cU}{\mathcal{U}}

\newcommand{\cQ}{\mathcal{Q}}

\newcommand{\bA}{\mathbf{A}}
\newcommand{\bL}{\mathbf{L}}

\newcommand{\bP}{\mathbf{P}}

\newcommand{\bH}{\mathbf{H}}

\newcommand{\A}{\mathbb{A}}

\newcommand{\cJ}{\mathcal{J}}
\newcommand{\boldw}{\boldsymbol{w}}
\newcommand{\bolds}{\boldsymbol{s}}

\author{Guillaume Baverez}
\address{Institut de Math\'ematiques de Marseille, Aix--Marseille Universit\'e.}
\email{guillaume.baverez@gmail.com}

\author{Baojun Wu}
\address{Beijing International Center for Mathematical Research (BICMR) }
\email{wubaojunmathe@outlook.com}

\title[Irreducible representations in LCFT]{Irreducible Virasoro representations\\ in Liouville conformal field theory}


\begin{document}	
	\begin{abstract}
This paper studies the analytic continuation of Liouville eigenstates and shows that they assemble into irreducible highest-weight representations of the Virasoro algebra, for all values of the conformal weights. This builds on previous results from the first author and Guillarmou, Kupiainen, Rhodes \& Vargas, where such representations were constructed except for the conformal weight on the Kac table. In order to extend these results to the degenerate weights, we find explicit analytic expressions for the Virasoro descendants and uncover the probabilistic meaning of the Kac table.

In the algebraic approach to conformal field theory, the irreducibility is a crucial property that must be satisfied by the representations in the spectrum, and is usually taken as an axiom. Computationally, it leads to the celebrated null-vector (or BPZ) equations for correlation functions and conformal blocks, which are the cornerstone of the integrability of the theory. 
\end{abstract}
	
	\maketitle
	\setcounter{tocdepth}{1}
	\tableofcontents

\section{Introduction}

	\subsection{Motivation and background}\label{subsec:motivation}
The history of two-dimensional statistical mechanics at criticality is dotted with striking predictions from physics which required a wealth of mathematical techniques to find a rigorous proof. Notable examples include Cardy's crossing formula in percolation \cite{Cardy_crossing,Smirnov_crossing}, and the connective constant of the honeycomb lattice \cite{Nienhuis1,Nienhuis2,DuminilSmirnov}. The most puzzling aspect of these formulas is that their derivation seems completely disconnected from the underlying probabilistic model: rather, they rely on algebraic methods coming from the representation theory of an infinite dimensional Lie algebra called the Virasoro algebra. What justifies these methods is the conjecture that critical systems have conformally invariant scaling limits \cite{Polyakov70}, and the Virasoro algebra is the corresponding symmetry algebra. This principle led to the groundbreaking work of Belavin--Polyakov--Zamolodchikov (BPZ) \cite{BPZ84}, laying the foundations of two-dimensional conformal field theory (CFT) and giving predictions for critical exponents of many statistical mechanical systems (the ``minimal models" which include the Ising model). 

In general, it is far from obvious how to relate the algebraic approach to the probabilistic one, and this article establishes one such link in the context of the probabilistic formulation of Liouville CFT \cite{DKRV16}. The exact solvability of this theory takes the form of an analytic formula for the three-point correlation function on the sphere (a.k.a. the structure constant): this is the celebrated ``DOZZ formula", derived in \cite{DornOtto94,ZZ95}, supported by Teschner \cite{Teschner_DOZZ}, and proved in \cite{KRV_DOZZ}. This proof (as well as the mating-of-trees approach to Liouville theory \cite{MatingOfTrees}) inspired tremendous developments in probability theory, with recent proofs of exact formulas for boundary Liouville CFT \cite{ARS25_FZZ,ARSZ23} and the conformal loop ensemble \cite{ACSW1,ACSW2}.
At the core of the proof of the DOZZ formula are the famous BPZ equations, which are PDEs satisfied by certain correlation functions of the theory. While these PDEs have been established rigorously in the probabilistic setting \cite{KRV19_local}, their origin is representation theoretic in nature, and one of the goals of this paper is to bring this hidden algebraic structure into the probabilistic world.

First steps in this direction have recently been made: in \cite{GKRV20_bootstrap}, the authors define a Hilbert space $\cH$ acted on by a positive, self-adjoint operator $\bH$ (the \emph{Liouville Hamiltonian}), which they diagonalise completely. In \cite{BGKRV22}, together with the first named author, they define an infinite family of (unbounded) operators $(\bL_n,\tilde{\bL}_n)_{n\in\Z}$ on $\cH$ which form two commuting representations of the Virasoro algebra, and generalise the Hamiltonian in the sense that $\bH=\bL_0+\tilde{\bL}_0$. This construction implies that the eigenstates of $\bH$ carry some information of both analytic and algebraic flavour; in particular, they organise into families of \emph{highest-weight representations} \cite[Theorem 4.5]{BGKRV22}, parametrised by a scalar called the \emph{conformal weight}. Each such representation has a distinguished state called the \emph{primary state}, and the others are the \emph{descendants}.

The algebraic origin of the BPZ equations lies in the assumption that the representations involved are \emph{irreducible}. For certain values of the conformal weight (the \emph{Kac table} \eqref{eq:kac}, labelling the \emph{degenerate representations}), this gives a crucial piece of information: certain states in the representation satisfy non-trivial linear relations (see Example \ref{ex:level_two} for a concrete instance of this phenomenon). This information, together with the \emph{conformal Ward identities}, leads eventually to the BPZ equations. The Ward identities are beyond the scope of this article, but the reader may think of them as a black box turning a descendant state into a partial differential operator; hence a linear relation between descendant states translates into a PDE. While the Ward identities and BPZ equations have been established rigorously in the probabilistic language \cite{KRV19_local,Zhu20_BPZ,GKRV20_bootstrap}, it would have been extremely hard to guess them without the algebraic input. It is therefore of great importance to get a firm probabilistic grasp on these algebraic objects. Our main result makes a new step in this direction by proving the irreducibility of the degenerate representations using the probabilistic formulation of Liouville CFT (Theorem~\ref{thm:singular}), thus answering an important question on the structure of the spectrum of the theory.

 The physical motivation for the axiom of irreducibility is that reducible representations include non-zero states of vanishing norm\footnote{with respect to the Shapovalov form (Appendix \ref{app:virasoro}). For the expert reader, the Shapovalov form can be expected to coincide with the $L^2$-product (i.e. the ``physical" inner-product) for unitary theories only. See \cite{BJ24,GQW25} for an example of non-unitarity for the loop version of Schramm--Loewner evolutions.} (a.k.a. \emph{singular vectors}), which are deemed unphysical. Of course, this is not a mathematical argument, and this axiom should in fact be proved when studying a concrete CFT. However, it seems like this axiom has never been questioned in physics until only relatively recently. In fact, new developments in physics and mathematics have shown that the question of the nature of singular vectors (i.e. whether or not they vanish) deserves a great deal of attention. In physics, the question has sparked some interest in relation to loop models, statistical mechanics, and Liouville theory with central charge less than~1 \cite{JacobsenSaleur,JacobsenSaleur21,SantachiaraPicco,RibaultSantachiara}. In boundary Liouville CFT, the BPZ equations hold only if the cosmological constants satisfy suitable relations \cite{Ang23_zipper}, a sign that some singular vectors don't vanish otherwise. Indeed, this is related to the \emph{higher equations of motion}, predicted in \cite{Belavin2_bHEM} and established in \cite{BaverezWu_hem,Cercle_HEM}. The article \cite{BaverezWu_hem} also studied the bulk version of these equations, which had been conjectured earlier in \cite{Zamolodchikov03_HEM}. New versions of these equations appeared in a recent work on Toda theories \cite{CercleHuguenin}. In the distinct (but related) theory of Schramm--Loewner evolutions (more precisely the SLE loop measure \cite{Zhan21}), the nature of singular vectors was  elucidated by the first named author and Jego \cite{BJ24}. In the chordal version of SLE, the level-two BPZ equations have been linked early on to martingale observables, but the precise link with the field content of the theory is only partially understood (see \cite{Peltola_sleReview} for a comprehensive review of this vast literature). 

The remainder of this article is organised as follows. In Section \ref{subsec:results}, we state our main results, Theorems \ref{thm:poisson} and \ref{thm:singular}. Theorem \ref{thm:poisson} deals with the analytic continuation of the so-called Poisson operator \cite[Sections 6-7]{GKRV20_bootstrap}, which roughly speaking sends eigenstates of the free theory to eigenstates of the interacting (Liouville) theory. This theorem is the technical tool which allows us to prove Theorem \ref{thm:singular} on the structure of degenerate representations, as explained in Section \ref{subsec:singular} and the proof outline of Section \ref{subsec:proof_overview}. In Section~\ref{sec:background}, we give some background on Liouville CFT and Feigin--Fuchs modules. Section~\ref{sec:poisson} is dedicated to the proof of Theorem \ref{thm:poisson} (which implies Theorem \ref{thm:singular}). Appendix \ref{app:virasoro} records some useful facts and terminology about the Virasoro algebra and its representations.

	\subsection{Main results}\label{subsec:results}
	Here and in the sequel, $\gamma$ and $\mu$ are parameters satisfying
	\[\gamma\in(0,2)\qquad\text{and}\qquad\mu>0.\]
	We will sometimes allow $\mu$ to take complex values, provided $\Re(\mu)>0$. We also set,
	\[Q:=\frac{\gamma}{2}+\frac{2}{\gamma}>2;\qquad c_\mathrm{L}:=1+6Q^2>25.\]
	Given $r,s\in\N^*=\Z_{>0}$, we set
	\[\alpha_{r,s}:=(1-r)\frac{\gamma}{2}+(1-s)\frac{2}{\gamma},\]
	and for all $\alpha\in\C$:
	\[\Delta_\alpha:=\frac{\alpha}{2}\left(Q-\frac{\alpha}{2}\right).\]
	Observe that $\Delta_\alpha=\Delta_{2Q-\alpha}$. The \emph{Kac tables} are the discrete sets
\begin{equation}\label{eq:kac}	
	kac^-:=\{\alpha_{r,s}|\,r,s\in\N^*\}\subset(-\infty,0]\qquad\text{and}\qquad kac^+:=2Q-kac^-\subset[2Q,\infty).
	\end{equation}
	We also set $kac:=kac^-\cup kac^+$.

In the remainder of this section, we introduce our main results. In order to lighten the presentation, we give only the minimal amount of information required to state the results, and postpone the detailed background to Section \ref{sec:background}.
	
	\subsubsection{The Poisson operator}\label{subsubsec:poisson}
	
	Let $\cF=\C[(\varphi_n)_{n=1}^\infty]$ be the space of polynomials in countably many complex variables $\varphi_n$. Given a sequence $\mathbf{k}=(k_n)_{n=1}^\infty$ of non-negative integers with finitely many non-zero terms, we define the polynomial $\varphi^\mathbf{k}:=\prod_{n=1}^\infty\varphi_n^{k_n}\in\cF$, and we call $|\mathbf{k}|:=\sum_{n=1}^\infty nk_n$ the \emph{level} of $\mathbf{k}$ (and $\varphi^\mathbf{k}$). We have a grading of $\cF$ by the level
	\[\cF=\oplus_{N\in\N}\cF_N,\]
with $\cF_N$ spanned by the monomials at level $N$.

One can define an infinite measure $\d c\otimes\P_{\S^1}$ on the Sobolev space $H^{-1}(\S^1)$, and the \emph{Liouville Hilbert space} is $\cH=L^2(\d c\otimes\P_{\S^1})$. Here, $\P_{\S^1}$ is a Gaussian measure on the space of distributions on $\S^1$ with vanishing mean (with respect to the Haar measure on $\S^1$), and $\d c$ is the Lebesgue measure on $\R$. See Section \ref{subsec:hilbert_space} for details. Functions (or distributions) with vanishing mean on the circle have a Fourier expansion $\varphi(e^{i\theta})=\sum_{n\in\Z\setminus\{0\}}\varphi_ne^{ni\theta}$, with $\varphi_{-n}=\bar{\varphi}_n$. This allows us to identify $\cF$ as a linear subspace of $L^2(\P_{\S^1})$. 
	
	Acting on $\cH$, there are two positive, essentially self-adjoint operators $\bH^0$, $\bH$, called respectively the \emph{free field and Liouville Hamiltonians}. The free field Hamiltonian is given by 
\begin{equation}\label{eq:hamiltonian_ff}	
	\bH^0:=-\frac{1}{2}\partial_c^2+\frac{1}{2}Q^2+2\sum_{n=1}^\infty\bA_{-n}\bA_n+\widetilde{\mathbf{A}}_{-n} \widetilde{\mathbf{A}}_n,
	\end{equation}
	where for all $n\geq1$,
	\begin{align*}
		& \mathbf{A}_n=\frac{i}{2} \partial_n, \quad \mathbf{A}_{-n}=\frac{i}{2}\left(\partial_{-n}-2 n \varphi_n\right) \\
		& \widetilde{\mathbf{A}}_n=\frac{i}{2} \partial_{-n}, \quad \widetilde{\mathbf{A}}_{-n}=\frac{i}{2}\left(\partial_n-2 n \varphi_{-n}\right) \\
		& \mathbf{A}_0=\widetilde{\mathbf{A}}_0=\frac{i}{2}\left(\partial_c+Q\right) .
	\end{align*}
	Here $\partial_n:=\del_{\varphi_n}$ and $\partial_{-n}:=\del_{\bar\varphi_n}$ denote the complex derivatives in the variable $\varphi_n$, $n\geq1$.

	The Liouville Hamiltonian is the Schr\"odinger-type operator 
\begin{equation}\label{eq:liouville_hamiltonian}
	\bH=\bH^0+\mu e^{\gamma c}V(\varphi),
\end{equation}
 where the potential is formally $V(\varphi):=\int_{\S^1}e^{\gamma\varphi(\theta)}d\theta$, and is understood rigorously thanks to a Cameron--Martin formula \eqref{eq:potential}. Both operators $\bH^0$, $\bH$ generate contraction semi-groups $(e^{-t\bH^0})_{t\geq0}$, $(e^{-t\bH})_{t\geq0}$ acting on $\cH$, see Section \ref{subsec:semigroups}. The operator $\bH^0$ is easy to diagonalise \cite[Section 4]{GKRV20_bootstrap}, in particular the function $e^{(\alpha-Q)c}\varphi^\mathbf{k}$ is a (generalised) eigenstate of $\bH^0$ with eigenvalue $2\Delta_\alpha+|\mathbf{k}|$, for each $\alpha\in\C$ and sequence $\mathbf{k}$ as above.
	
	One of the main inputs of \cite{GKRV20_bootstrap} is to diagonalise $\bH$. It is shown in Sections 6-7 therein that the long-time asymptotic of the heat flow of $\bH$ can produce generalized eigenstates. More precisely, for all $\chi\in\cF_{|\mathbf{k}|}$ and $\alpha$ in a neighbourhood of $-\infty$ (depending on the level of $\chi$, see \cite[Proposition 7.2]{GKRV20_bootstrap}), we have the following limit in a weighted space $e^{-\beta c}\cD(\cQ)$:
	\[\cP_\alpha(\chi)=\underset{t\to\infty}\lim\,e^{t(2\Delta_\alpha+|\mathbf{k}|)}e^{-t\bH}\left(\chi e^{(\alpha-Q)c}\right),\]
	and the map $\alpha\mapsto\cP_\alpha(\chi)$ is analytic in this region. Here, $\cD(\cQ)\subset\cH$ is the domain of the quadratic form of $\bH$ (see Section \ref{subsec:semigroups}) and $\beta$ is a suitable real number. The region provided in \cite[Proposition 7.2]{GKRV20_bootstrap} gets smaller as the level of the polynomial increases. For each $\alpha$, $$\cP_\alpha=\oplus_{N\in\N}\cP_{\alpha,N}:\cF\to e^{-\beta c}\cD(\cQ)$$ is a linear map and is called the \emph{Poisson operator}, where $\cP_{\alpha,N}$ is the restriction of $\cP_\alpha$ to the (finite-dimensional) space $\cF_N$. 
	
	It is easy to deduce from \cite{GKRV20_bootstrap,BGKRV22} that the Poisson operator extends meromorphically to all $\alpha\in\C$ with values in $e^{-\beta c}\cD(\cQ)$ such that $\beta>Q-\Re(\alpha)$. Our first result is to discard the possibility of poles.
	\begin{theorem}\label{thm:poisson}
		For each $\chi\in\cF$, the map $\alpha\mapsto\cP_\alpha(\chi)$ is analytic on the whole $\alpha$-plane, i.e. it has no poles.
	\end{theorem}
	This theorem is proved in Section \ref{sec:poisson}, and an overview of the main steps of the proof is given in Section \ref{subsec:proof_overview}. We will explain in the next section how it is related to the structure of singular representations, and state our main result on these representations.

	\subsubsection{Liouville singular modules}\label{subsec:singular}
This section uses some representation theoretic terminology, which we have summarised in Appendix \ref{app:virasoro} for the reader's convenience. 

 For each $\alpha\in\C$, let $\bA^\alpha_0:=\frac{i}{2}\alpha$, and
\begin{equation}\label{eq:sugawara}	
	\begin{aligned}
		&\bL_n^{0,\alpha}:=i(\alpha-(n+1)Q)\bA_n+\sum_{m\neq\{0,n\}}\bA_{n-m}\bA_m,\qquad n\neq0;\\
		&\bL_0^{0,\alpha}:=\Delta_\alpha+2\sum_{m=1}^\infty\bA_{-m}\bA_m.
	\end{aligned}
\end{equation}
These operators are known as the \emph{Sugawara (or Feigin--Fuchs) construction} and represent the Virasoro algebra on $\cH$. See Section \ref{subsec:ff_modules} for details and references on this construction. 

	Given a non-increasing sequence of positive integers $\nu=(\nu_1,...,\nu_\ell)$, we write $\bL^{0,\alpha}_{-\nu}:=\bL_{-\nu_\ell}^{0,\alpha}...\bL_{-\nu_2}^{0,\alpha}\bL_{-\nu_1}^{0,\alpha}$. Let $\cV_\alpha^0\subset\cF$ be the linear span of the states $\bL_{-\nu}^{0,\alpha}\ind$, for all finite sequences $\nu$ of non-increasing integers. Here, the linear span is in the algebraic sense (finite linear combinations), and $\ind$ is the constant function 1. This is a highest-weight representation with central charge $c_\mathrm{L}$ and highest-weight $\Delta_\alpha$. The classification of such modules is well-known and reviewed in Section \ref{subsec:ff_modules}. Briefly, for $\Re(\alpha)<Q$, the module $\cV_{2Q-\alpha}^0$ is always Verma (hence $\cV_{2Q-\alpha}^0=\cF$), while $\cV_\alpha^0$ can be the quotient of the Verma by a non-trivial submodule (in which case $\cV_\alpha^0$ is a proper subspace of $\cF$). Thus, there is a canonical projection $\Phi_\alpha^0:\cV_{2Q-\alpha}^0\to\cV_\alpha^0$, and it is analytic as a function of $\alpha$ in the region $\Re(\alpha)<Q$ (as explained after the statement of Lemma \ref{lem:frenkel}). We then define $\Phi_\alpha:=\cP_\alpha\circ\Phi_\alpha^0:\cF\to e^{-\beta c}\cD(\cQ)$, and introduce the \emph{Liouville module}
	\begin{equation}\label{eq:def_module}
		\cV_\alpha:=\mathrm{ran}(\Phi_\alpha)\simeq\cF/\ker(\Phi_\alpha).
	\end{equation}
The range of $\Phi_\alpha$ is in a weighted space $e^{-\beta c}\cD(\cQ)$ for some $\beta>0$, and $\cD(\cQ)$ is the domain of the quadratic form of $\bH$. This is the space where the eigenstates of $\bH$ can naturally be defined. See Section \ref{subsec:semigroups} for details.
	
	If we don't assume Theorem \ref{thm:poisson}, this expression is only well-defined on the domain of analyticity of $\cP_\alpha$. In this region, for all $\chi\in\cF$, $\Phi_\alpha(\chi)$ can be expressed as a (finite) linear combination of generalised eigenstates of $\bH$ \cite[Section 6.9]{GKRV20_bootstrap}. Moreover, it was shown in \cite[Theorem 4.4.]{BGKRV22} that these eigenstates have an analytic continuation to the whole $\alpha$-plane, and satisfy a reflection formula that can be compactly stated as $\Phi_{2Q-\alpha}=R(\alpha)\Phi_\alpha$, where $R(\alpha)\in\C$ is the \emph{reflection coefficient} \cite[(4.4)]{BGKRV22} (this reflection formula expresses the fact that the scattering matrix is diagonal). Due to the analyticity of $\Phi_\alpha$, we get that $\cP_\alpha$ extends meromorphically to the region $\{\Re(\alpha)<Q\}$, and its poles may only occur on the zeros of $\Phi_\alpha^0$. Since $\Phi_\alpha^0$ is the canonical projection from the Verma module to $\cV_\alpha^0$, its kernel is a submodule of the Verma module, hence it may only be non-trivial on the \emph{Kac table} $kac^-$ \eqref{eq:kac} (see also Appendix \ref{app:virasoro}). Thus, $\alpha\mapsto\cP_\alpha(\chi)$ is meromorphic on $\C$ with possible poles contained in $kac^-$, for all $\chi\in\cF$.
	
	The next result is that $\cV_\alpha$ is the highest-weight Virasoro representation of $(\bL_n)_{n\in\Z}$, and we give a complete classification.  Here, the Virasoro representation $(\bL_n)_{n\in\Z}$ on $\cV_\alpha$ is the one constructed probabilistically in \cite[Theorem 1.1]{BGKRV22}, and the Poisson operator has the property of intertwining the free field representation $(\bL_n^{0,\alpha})$ with the Liouville one
	$$\cP_\alpha\circ\bL_n^{0,\alpha}=\bL_n\circ\cP_\alpha$$

	The next statement uses some notation and terminology from Appendix~\ref{app:virasoro}. In particular, $M(c_\mathrm{L},\Delta)$ denotes the Verma module of charge $c_\mathrm{L}$ and weight $\Delta$, and $V(c_\mathrm{L},\Delta)$ is its irreducible quotient. 
		
	\begin{theorem}\label{thm:singular}
		For $\alpha\in \C$, the space $\cV_\alpha$ defined in \eqref{eq:def_module} is the highest-weight representation of the Virasoro algebra. Moreover, it has the following structure:
		\begin{itemize}
			\item For $\alpha\notin kac$, $\cV_\alpha$ is isomorphic to the Verma module $M(c_\mathrm{L},\Delta_\alpha)$ and it is irreducible.
			\item For $\alpha\in kac^-$, $\cV_\alpha$ is isomorphic to the irreducible quotient $V(c_\mathrm{L},\Delta_\alpha)$ of the Verma module $M(c,\Delta_\alpha)$ by the maximal proper submodule.
			\item For $\alpha\in kac^+$, $\cV_\alpha$ is isomorphic to $V(c_\mathrm{L},\Delta_\alpha)$ when $\alpha \notin(Q+\frac{2}{\gamma}\N^*)\cup (Q+\frac{\gamma}{2}\N^*)$ and $\cV_\alpha =0$ when $\alpha \in(Q+\frac{2}{\gamma}\N^*)\cup (Q+\frac{\gamma}{2}\N^*)$.
		\end{itemize}
	\end{theorem}
	
The relation between Theorem \ref{thm:singular} and Theorem \ref{thm:poisson} is explained in Section~\ref{subsec:proof_overview}, and both results will be proved in Section \ref{sec:poisson}. Only the second item of Theorem \ref{thm:singular} is the technical subject of this paper. Indeed, the first item was done in \cite[Theorem 4.5]{BGKRV22}, while the third corresponds to the reflection formula $\cV_{2Q-\alpha}=R(\alpha)\cV_\alpha$ stated above. Hence, $\cV_{2Q-\alpha}\simeq\cV_\alpha$ (as highest-weight representations) unless $R(\alpha)=0$, in which case $\cV_{2Q-\alpha}=0$. The zeros of $R(\alpha)$ occur on $(Q+\frac{\gamma}{2}\N^*)\cup(Q+\frac{2}{\gamma}\N^*)$ \cite[(4.4)]{BGKRV22}, explaining this range. In the sequel, we will focus on proving the second item.
	
	The relevance of this Theorem \ref{thm:singular} is two-fold. On the fundamental level, it gives a complete classification of the algebraic structure of Liouville CFT. On the practical level, we get non-trivial relations between Liouville eigenstates for $\alpha$ in the Kac table, which are precisely the relations that are exploited through the BPZ equations. In a highest-weight representation, a \emph{singular vector} is one which is both a descendant and a highest-weight vector. In particular, any non-zero singular vector generates a non-trivial proper submodule. Theorem \ref{thm:singular} implies that for each $\alpha_{r,s}=Q-\frac{\gamma}{2}r-\frac{2}{\gamma}s$, $\cV_{\alpha_{r,s}}$ does not have any nonzero singular vector. On the other hand, the corresponding Verma module does have singular vectors, so the canonical projection maps singular vectors to zero. With the notation from the previous section, this means that there exists $\chi\in\cF_{rs}$ such that $\Phi_{\alpha_{r,s}}(\chi)=0$. Thus, there is a non-trivial linear combination 
	\begin{equation}\label{eq:singular_operator} 
		\bS_{\alpha_{r,s}}=\sum_{|\nu|=rs}\sigma_\nu^{r,s}\bL_{-\nu}\qquad\text{such that}\qquad\bS_{\alpha_{r,s}}\Psi_{\alpha_{r,s}}=0.
	\end{equation}
	Moreover, the coefficients $(\sigma_\nu^{r,s})_{|\nu|=rs}$ are universal in the sense that they only depend on $c_\mathrm{L}$ and $\alpha_{r,s}$. The singular vector at a given level is unique up to constant, and this constant can be fixed such that $\sigma^{r,s}_{(1,1,...,1)}=1$ \cite[Theorem 3.1]{Astashkevich97}. At level two, an explicit computation can be made, see Example~\ref{ex:level_two}.
	
	We stress that it would be quite difficult to prove such relations by analytic means, even for the level 2 above case, which was proved in \cite{KRV19_local}. In principle, one can find explicit formulae for the descendants (see the proof of Theorem \ref{thm:poisson}), but the combinatorics soon become intractable. Instead, Theorem \ref{thm:singular} gives us all the relations at once for all $r,s\in\N^*$.

	\subsection{Proof outline}\label{subsec:proof_overview}
	The proofs of Theorems \ref{thm:poisson} and \ref{thm:singular} use a combination of probabilistic and algebraic arguments which we summarise here. First, we stress that the two results are intimately related. Indeed, from the existence of a Virasoro representation on the Liouville side \cite{BGKRV22}, it is fairly easy to see that $\ker\Phi_{\alpha_{r,s}}$ is a Virasoro submodule for all $r,s\in\N^*$, so that $\cV_{\alpha_{r,s}}$ is a highest-weight representation as a quotient by a submodule. This immediately implies $\ker\Phi_{\alpha_{r,s}}\subset\ker\Phi_{\alpha_{r,s}}^0$ since the latter is the maximal proper submodule (\cite{Frenkel92_determinant} and Section \ref{subsec:proof_singular}). The converse inclusion is equivalent to Theorem \ref{thm:poisson}. 
	
	Now, we explain the main steps of the proof of Theorem \ref{thm:poisson}.
	\begin{enumerate}[label={\arabic*.}]
		\item Using Gaussian integration by parts/Girsanov's theorem, we exhibit a probabilistic formula for the descendant states, valid for $\alpha$ in a neighbourhood of $-\infty$ (Lemma \ref{lem:expression_descendants}). This formula is essentially a linear combination of integrals of the form $\cI_{r,\bolds}(\alpha)$ as defined in \eqref{eq:def_I}. These are disc amplitudes with an $\alpha$-insertion at 0 and several $\gamma$-insertions integrated over the unit disc. This reduces the question to the analytic continuation of these integrals, and their study on the Kac table.
		\item The integrals have singularities at zero, which is why they converge absolutely only for small enough $\alpha$. Using the fusion estimate of Proposition \ref{prop:regular_rectangle}, we get a sharp bound for the region of absolute convergence and analyticity in $\alpha$ of these integrals. 
		\item \label{item:mero} We find two induction formulae \eqref{eq:init_rec},\eqref{eq:rec}, expressing a singular integral in terms of simpler ones. This allows us to study the pole structure of the singular integrals recursively. As a byproduct, the induction formulae highlight the special role played by the Kac table.
		\item \label{item:four} From the first three inputs, we deduce that for each $r,s\in\N^*$, the Poisson operator is analytic at $\alpha_{r,s}$ and levels $\leq rs$, \emph{provided} that $\gamma^2\not\in\Q$ (Corollary \ref{cor:level_rs}). 
		\item\label{item:five} Using the Virasoro representation constructed in \cite{BGKRV22}, the analyticity propagates to all levels (Proposition \ref{prop:regular_descendants}). This implies that the Poisson operator $\alpha\mapsto\cP_\alpha(\chi)$ is analytic for \emph{all} $\alpha$ and \emph{all} $\chi\in\cF$, \emph{provided} $\gamma^2\not\in\Q$.
		\item\label{item:six} We exploit a continuity property in $\gamma$ (Proposition \ref{prop:gamma_continuous}) to extend the result from $\gamma^2\not\in\Q$ to arbitrary $\gamma\in(0,2)$. This is done in Section~\ref{subsec:proof_singular}. 
	\end{enumerate} 
	
The objects encountered in the proof bear some similarities with the ones in \cite{Oikarinen19_smooth}, which proves smoothness of Liouville correlation functions in the location of the insertions (what corresponds to our $w_j$'s below). Our goal, on the other hand, is to study the analytic continuation of the eigenstates in~$\alpha$. The construction of the \emph{meromorphic} continuation (Item \ref{item:mero}, Section~\ref{subsec:mero}) requires the careful study of singular integrals having the same flavour as in \cite{Oikarinen19_smooth}, and uses quite heavily the symmetrisation trick and the (first) derivative formula for correlation functions, already present in \cite{KRV19_local,Oikarinen19_smooth}. The crux (and main novelty) of our proof is to discard the possibility of poles, which require completely different techniques (as summarised in Items \ref{item:four}-\ref{item:six} above).

\subsection*{Acknowledgements}
	We thank Joona Oikarinen for discussions on the smoothness of Liouville correlation functions. We are also grateful to Colin Guillarmou and R\'emi Rhodes helpful discussions. We thank the anonymous referees for their thorough review which drastically improved the manuscript. G.B. is supported by ANR-21-CE40-0003 ``Confica". B.W. \ is supported by National Key R\&D Program of China (No. 2023YFA1010700).
	
	\section{Background}\label{sec:background}

 \subsection{Integer partitions and Young diagrams}\label{subsec:partitions}
Let $N\in\N$. A \emph{partition of $N$} is a way of writing $N$ as a sum of positive integers, where two sums are identified if they differ only by the order of their summands (for instance, $3+1$ and $1+3$ define the same partition of $4$). We denote by $\cT_N$ the set of partitions of $N$, and by $p(N)$ its cardinality. By convention, $p(0)=0$. We also set $\cT:=\cup_{n\in\N}\cT_N$ the set of all integer partitions.

Choosing to order the summands in non-decreasing order, a partition of $N$ is specified uniquely by a finite, non-decreasing sequence of positive integers $\nu=(\nu_1,...,\nu_\ell)$ such that $\sum_{j=1}^\ell\nu_j=N$. Such a $\nu$ is called a \emph{Young diagram} of \emph{length $\ell$} and \emph{level $N$}. Equivalently, a partition of $N$ is specified uniquely by the number of times $k_n$ that the integer $n$ appears in the partition. Hence, we can also represent the partition of an integer by a sequence $\mathbf{k}=(k_n)\in\N^{\N^*}$ with finitely many non-zero terms. The level of the corresponding partition $\sum_{n=1}^\infty nk_n$, and its length is $\sum_{n=1}^\infty k_n$. We will navigate freely between these two ways of parametrising integer partitions, with the model chosen being clear from the notation. We note that these two ways already appeared in Section \ref{subsec:results} when defining $\varphi^\mathbf{k}$ and $\bL_{-\nu}^{0,\alpha}$.

One of the key insights of our work is to relate the parameters $(r,s)$ of the Kac table \eqref{eq:kac} to the singular integrals defined in \eqref{eq:def_I}, which are also parametrised by integer partitions. To make the connection with the Kac table notationally explicit, the integrals $\cI_{r,\bolds}$ are parametrised by integer partitions $(r,\bolds)$, where $r$ is the length of the partition, and $\bolds=(s_1,...,s_r)$ is a Young diagram. These partitions will be used to parametrise singularities at the origin of the complex plane, namely to $(r,\bolds)$ we associate the meromorphic function of $r$ variables $(w_1,...,w_r)\mapsto\prod_{j=1}^rw_j^{-s_j}$ on $\D^r$ (see \eqref{eq:def_I}). It is pictorially convenient to represent a Young diagram in the ``English notation" as in Figure \ref{fig:elementary_move}, where the $j^{\text{th}}$ column represents a singularity $w_j^{-s_j}$.

Given $r,s\in\N^*$, we use the shorthand notation $(r,(s))$ for the ``rectangular" partition $(r,(s,s,...,s))$ of length $r$. Given a partition $(r,\bolds)$, we define the number $r^*$ (depending on $(r,\bolds)$) by 
	\[r^*:=\max\{j\in\{1,...,r\}|\,s_j=s_1\}.\]
We will keep the dependence on the partition implicit, creating no confusion in the future.
 
 
	 Now, we define the $j^{\text{th}}$ \emph{elementary move} $\tau_j:\cT_N\to\cT_N$, which consists informally in taking a box from the $j^{\text{th}}$ column and move it to the first available column to the right, in such a way that we still get a partition. Formally, the algorithm is as follows: for all $(r,\bolds)\in\cT_N$ and $j\leq r$,
	 \begin{itemize}
	 	\item If $s_j=1$, $\tau_j(r,\bolds)=(r,\bolds)$.
	 	\item If $s_j\geq2$, let $j^*:=\max\{l|\,s_l=s_j\}$. 
	 	\begin{itemize}
	 		\item If $s_j\in\{s_r,s_r+1\}$, then $\tau_j(r,\bolds)$ is the partition $(r+1,\tilde{\bolds})$, with $\tilde{s}_{j^*}=s_{j^*}-1$, $\tilde{s}_{r+1}=1$, and $\tilde{s}_l=s_l$ for $l\not\in\{j^*,r+1\}$.
	 		\item If $s_j\geq s_r+2$, let $k:=\min\{l|\,s_l\leq s_j-2\}$. Then $\tau_j(r,\bolds)$ is the partition $(r,\tilde{\bolds})$ with $\tilde{s}_{j^*}=s_j-1$, $\tilde{s}_k=s_k+1$, and $\tilde{s}_l=s_l$ for $l\not\in\{k,j^*\}$.
	 	\end{itemize}
	 \end{itemize}

	\begin{figure}[h!]
		\centering
		\includegraphics[scale=0.7]{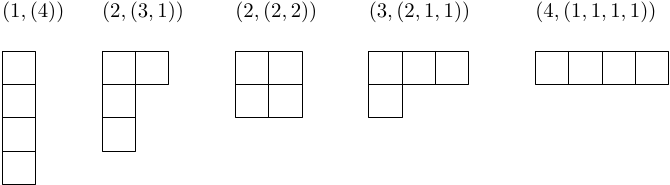}
		\caption{\label{fig:elementary_move} From left to right, a sequence of elementary moves in $\cT_4$.}
	\end{figure}

 We define a partial order $\preceq$ on $\cT_N$ by declaring $(\tilde{r},\tilde{\bolds})\preceq(r,\bolds)$ if $(\tilde{r},\tilde{\bolds})$ can be obtained from $(r,\bolds)$ by a (possibly empty) sequence of elementary moves. The associated strict order is denoted by $\prec$. It is elementary to check that the unique minimal (resp. maximal) partition in $\cT_N$ is $(N,(1))$ (resp. $(1,(N))$), and for each $(r,\bolds)\in\cT_N$, there exists a sequence of elementary moves from $(r,\bolds)$ to $(1,(N))$. The distance of a partition $(r,\bolds)$ to $(N,(1))$ is the minimal number of elementary moves required to reach $(N,(1))$ from $(r,\bolds)$, and it is a finite number. By definition, an elementary move decreases strictly the distance to the minimal partition (unless the partition is already the minimal one). By an induction on $\cT_N$, we mean a proof by induction on the distance to the minimal partition.
	
	\subsection{Liouville CFT}\label{subsec:hilbert_space}
	
This section gives some details on the material introduced in Section \ref{subsubsec:poisson}, i.e. we review the Liouville CFT, its Hilbert space and Hamiltonian. Our review is based on \cite[Sections 3-5]{GKRV20_bootstrap}.

		\subsubsection{Gaussian fields and Liouville Hilbert space}
	Let $\P_{\S^1}$ be the law of the centred, Gaussian, $\log$-correlated field on the circle, i.e. samples from $\P_{\S^1}$ are Gaussian fields $\varphi$ on the circle with covariance 
	\[\E[\varphi(e^{i\theta})\varphi(e^{i\theta'})]=\log\frac{1}{|e^{i\theta
		}-e^{i\theta'}|},\qquad\forall\theta,\theta'\in\R.\]
	With $\P_{\S^1}$-probability 1, $\varphi\in H^{-s}(\S^1)$ for all $s>0$. In Fourier modes, one has
	\[\varphi(e^{i\theta})=\sum_{n\in\Z\setminus\{0\}}\varphi_ne^{ni\theta},\]
	where $\varphi_{-n}=\bar{\varphi}_n$, and $(\varphi_n)_{n\geq1}$ are independent complex Gaussians $\cN_\C(0,\frac{1}{2n})$. The harmonic extension $P\varphi\in C^\infty(\D)$ is 
	\[P\varphi(z)=\sum_{n=1}^\infty\varphi_nz^n+\varphi_{-n}\bar{z}^n,\qquad\forall z\in\D.\] 
	Its covariance structure is
	\begin{equation}\label{eq:covar_harmonic}
		\E[P\varphi(z)P\varphi(w)]=\log\frac{1}{|1-z\bar{w}|}=:G_\del(z,w).
	\end{equation}
	In particular, we have $\E[P\varphi(z)^2]=\log\frac{1}{1-|z|^2}$. The \emph{Liouville Hilbert space} is 
	\[\cH:=L^2(\d c\otimes\P_{\S^1}),\]
	where $\d c$ is the Lebesgue measure on $\R$. Samples from the (infinite) measure $\d c\otimes\P_{\S^1}$ are written $c+\varphi$. Following \cite[Section 4.1]{GKRV20_bootstrap}, we define a dense subspace $\cC$ of $L^2(\P_{\S^1})$ as follows. A functional $F\in C^0(H^{-1}(\S^1))$ belongs to $\cC$ if there exists $N\in\N$ and $f\in C^\infty(\C^N)$ such that $F(\varphi)=f(\varphi_1,...,\varphi_N)$ for all $\varphi\in H^{-1}(\S^1)$, and $f$ and all its derivatives have at most exponential growth at $\infty$. The space $\cC$ comes with a natural Fr\'echet topology, and we denote by $\cC'$ its topological dual, i.e. $\cC'$ is the space of continuous linear forms on $\cC$. We will refer to $\cC$ (resp. $\cC'$) as the space of \emph{test functions} (resp. \emph{distributions}).
	
	We will frequently introduce some extra randomness in the form of an independent Dirichlet free field in $\D$. This is a centred Gaussian field $X_\D$ with covariance structure
	\begin{equation}\label{eq:covar_dirichlet}
		\E[X_\D(z)X_\D(w)]=\log\left|\frac{1-z\bar{w}}{z-w}\right|=:G_\D(z,w).
	\end{equation}
	We will always denote
	\[X:=X_\D+P\varphi,\]
	which is a $\log$-correlated Gaussian field in $\D$:
	\begin{equation}\label{eq:covar_log}
		\E[X(z)X(w)]=\log\frac{1}{|z-w|}=:G(z,w).
	\end{equation}
	
	Let $\rho\in C^\infty(\C)$ be a radially symmetric bump function, with total integral~$1$. For $\epsilon>0$, we write $\rho_\epsilon(z):=\epsilon^{-2}\rho(\frac{z}{\epsilon})$, and $X_\epsilon:=X\star\rho_\epsilon$. The \emph{Gaussian multiplicative chaos} (GMC) measure of $X$ with parameter $\gamma\in(0,2)$ is the following weak limit of measures
	\[\d M_\gamma(z):=\underset{\epsilon\to0}\lim\,\epsilon^\frac{\gamma^2}{2}e^{\gamma X_\epsilon(z)}|\d z|^2,\]
	where the convergence holds in probability. This limit exists, is non-trivial, and is independent of the choice of regularisation. GMC was introduced pioneered by Kahane in the context of turbulence \cite{Kahane85}, and we refer to \cite{rhodes2014_gmcReview} for an account of recent developments. In the form just stated, \cite{Berestycki17} gives a concise and comprehensive introduction to the topic (see in particular \cite[Theorem 1.1]{Berestycki17} for the existence and uniqueness of $M_\gamma$).
	
	\subsubsection{Semigroups}\label{subsec:semigroups}
	One can define two strongly continuous, contraction semigroups on $\cH$, with (unbounded) generators $\bH^0$, $\bH$ \cite{GKRV20_bootstrap}. These generators are positive, essentially self-adjoint, unbounded operators on $\cH$. 
	
	According to \cite[Proposition 4.4]{GKRV20_bootstrap}, the operator $\bH^0$ from \eqref{eq:hamiltonian_ff} generates a semigroup $(e^{-t\bH^0})_{t\geq0}$ with the expression
	\[e^{-t\bH^0}F(c+\varphi):=e^{-\frac{Q^2}{2}t}\E_\varphi\left[F(c+B_t+\varphi_t)\right],\qquad\forall F\in\cH,\]
	where $B_t=\int_0^{2\pi}X(e^{-t+i\theta})\frac{\d\theta}{2\pi}$, and $\varphi_t(e^{i\theta})=X(e^{-t+i\theta})-B_t$. The notation $\E_\varphi$ means conditionally on $\varphi$ (i.e. we integrate over the Dirichlet free field $X_\D$), and $(B_t)_{t\geq0}$ evolves as a standard Brownian motion. Moreover, the Fourier modes of $\varphi_t$ evolve as independent Ornstein-Uhlenbeck processes. See \cite[Section 4.3]{GKRV20_bootstrap} for details on this semigroup.
	
	Similarly, the Liouville Hamiltonian $\bH$ from \eqref{eq:liouville_hamiltonian} exponentiates to a contraction semigroup $(e^{-t\bH})_{t\geq0}$ on $\cH$, as we now review. First, the potential $V$ is defined in $\cC'$ by the formula
\begin{equation}\label{eq:potential}
\E[V(\varphi)F(\varphi)]=\int_0^{2\pi}\E[F(\varphi-\gamma\log|e^{i\theta}-\cdot|)]\d\theta,\qquad\forall F\in\cC,
\end{equation}
which, by the Cameron--Martin formula, makes rigorous sense of the formal expression $V(\varphi)=\int_0^{2\pi}e^{\gamma\varphi(e^{i\theta})-\frac{\gamma^2}{2}\E[\varphi(e^{i\theta})^2]}\d\theta$. This defines $\bH=\bH^0+\mu e^{\gamma c}V$ as a positive, self-adjoint (unbounded) operator on $\cH$. It is shown in \cite[Section 5]{GKRV20_bootstrap} that this operator exponentiates to a contraction semigroup with the expression \cite[(5.3)]{GKRV20_bootstrap}
	\begin{equation}\label{eq:liouville_semigroup}
		e^{-t\bH}F(c+\varphi)=e^{-\frac{Q^2}{2}t}\E_\varphi\left[F(c+B_t+\varphi_t)e^{-\mu e^{\gamma c}\int_{\A_t}\frac{\d M_\gamma(z)}{|z|^{\gamma Q}}}\right],\qquad\forall F\in\cH,
	\end{equation}
	where $\A_t:=\{z\in\C|\,e^{-t}<|z|<1\}$. As explained in \cite[Section 5.1]{GKRV20_bootstrap}, this expression can be understood as a Feynman--Kac formula for the Schr\"odinger operator $\bH=\bH^0+\mu e^{\gamma c}V$, but this formula is far from obvious due to the irregular nature of the potential. The identification of $\bH$ with the generator of the semigroup \eqref{eq:liouville_semigroup} is a deep question exploiting many subtle properties of GMC \cite[Section 5]{GKRV20_bootstrap} (see in particular Proposition 5.5 there). The strategy adopted in \cite{GKRV20_bootstrap} is to first construct a quadratic form $\cQ$ with dense domain $\cD(\cQ)\subset\cH$, and study its properties \cite[Lemma 5.4]{GKRV20_bootstrap}. Then, $\bH$ is constructed as the Friedrichs extension of $\cQ$ \cite[Proposition 5.5]{GKRV20_bootstrap}. With these definitions, $\bH$ defines a bounded operator $\bH:\cD(\cQ)\to\cD'(\cQ)$, where $\cD'(\cQ)$ is the continuous dual of $\cD(\cQ)$ \cite{GKRV20_bootstrap} (see the last displayed equation from Section 5 there). Of course, the same applies to weighted space, i.e. $\bH:e^{-\beta c}\cD(\cQ)\to e^{-\beta c}\cD'(\cQ)$ defines a bounded operator for all $\beta\in\R$. The weighted spaces $e^{-\beta c}\cD(\cQ)$ are natural spaces in which the generalised eigenstates of $\bH$ are defined \cite[Theorem 4.5]{BGKRV22}.
	
	\subsection{Feigin--Fuchs modules}\label{subsec:ff_modules}
	
In this section, we give a detailed background on the Feigin--Fuchs modules (a.k.a. the Sugawara construction) introduced in Section \ref{subsec:singular}. These are standard constructions from representation theory, the structure of which is well known since the work of Frenkel \cite{Frenkel92_determinant}. We will follow the presentation (and notation) of \cite[Section 4.4]{GKRV20_bootstrap}. Another standard reference is \cite[Sections 2.3 and 3.4]{KacRaina_Bombay}, where it is called the ``oscillatory representation". We also use some notation and terminology from Appendix \ref{app:virasoro}, see in particular the notion of highest-weight representation and Verma module (denoted $M(c_\mathrm{L},\Delta_\alpha)$) introduced there. 
	
	Let $\alpha\in\C$ be arbitrary. We have a representation of the Heisenberg algebra $(\bA_n^\alpha)_{n\in\Z}$ as unbounded operators on $L^2(\P_{\S^1})$, given for $n>0$ by
	\[\bA_n^\alpha=\frac{i}{2}\del_n;\qquad\bA_{-n}^\alpha=\frac{i}{2}(\del_{-n}-2n\varphi_n);\qquad\bA_0^\alpha=\frac{i}{2}\alpha,\]
where we recall that $\del_n=\del_{\varphi_n}$ and $\del_{-n}=\del_{\bar\varphi_n}$ denote the complex derivatives in the variable $\varphi_n$. For $n\neq0$, we will simply write $\bA_n=\bA_n^\alpha$. We recall the main properties of these operators and refer to \cite[Section 4.4.1]{GKRV20_bootstrap} for details. These operators preserve $\cC$ and satisfy the commutation relations $[\bA_n^\alpha,\bA_m^\alpha]=\frac{n}{2}\delta_{n,-m}$ there. For $n\neq0$, we have the hermiticity relations $\bA_n^*=\bA_{-n}$ on $L^2(\P_{\S^1})$. Given a partition $\mathbf{k}$, we set $\bA_{-\mathbf{k}}:=\prod_{n=1}^\infty\bA_{-n}^{k_n}$. Note that the operators $\bA_{-n}$ all commute for $n>0$, so the order in which we do this product is irrelevant. 
	
	We call the constant function $\ind\in L^2(\P_{\S^1})$ the \emph{vacuum vector}. The Heisenberg representation gives the space $\cF=\C[(\varphi_n)_{n\geq1}]$ a structure of highest-weight Heisenberg module, i.e. 
	\[\cF=\mathrm{span}\,\{\pi_\mathbf{k}:=\bA_{-\mathbf{k}}\ind|\,\mathbf{k}\in\cT\}.\]
	The level gives a grading 
	\[\cF=\oplus_{N\in\N}\cF_N,\]
	and we have $\dim\cF_N=p(N)$. We will say that an element of $\cF_N$ is a polynomial of level $N$.
	
	Recall the Sugawara construction $(\bL_n^{0,\alpha})_{n\in\Z}$ from \eqref{eq:sugawara}. These operators preserve $\cC$ and are closable on $L^2(\P_{\S^1})$ \cite[Section 4.4]{GKRV20_bootstrap}. On $\cC$, they satisfy the commutation relations of the Virasoro algebra:
	\[[\bL_n^{0,\alpha},\bL_m^{0,\alpha}]=(n-m)\bL_{n+m}^{0,\alpha}+\frac{c_\mathrm{L}}{12}(n^3-n)\delta_{n,-m}.\]
	We also have the hermiticity relations $(\bL_n^{0,\alpha})^*=\bL_{-n}^{0,2Q-\bar{\alpha}}$ on $L^2(\P_{\S^1})$. As a module over the Virasoro algebra, $(\cF,(\bL_n^{0,\alpha})_{n\in\Z})$ is called a \emph{Feigin--Fuchs module}.
	

	Given a partition $\nu=(\nu_1,...,\nu_\ell)$, we set $\bL_{-\nu}^{0,\alpha}=\bL_{-\nu_\ell}^{0,\alpha}...\bL_{-\nu_1}^{0,\alpha}$. The \emph{descendant state}
	\[\cQ_{\alpha,\nu}:=\bL_{-\nu}^{0,\alpha}\ind\in\cF\]
	is a polynomial of level $|\nu|$. A straightforward computation using \eqref{eq:sugawara} shows that $\bL_n^{0,\alpha}\ind=0$ for all $n\geq1$, and $\bL_0^{0,\alpha}\ind=\Delta_\alpha\ind$, i.e. the constant function is a highest-weight state. By definition, this means that the space
	\[\cV_\alpha^0:=\mathrm{span}\{\cQ_{\alpha,\nu}|\nu\in\cT\}\]
is a highest-weight representation (with central charge $c_\mathrm{L}$ and weight $\Delta_\alpha$). We also define $\cV_\alpha^{0,N}:=\cV_\alpha^0\cap\cF_N$. For $\alpha\not\in kac$, the Verma module $M(c_\mathrm{L},\Delta_\alpha)$ is irreducible, so we have $\cV_\alpha^0\simeq M(c_\mathrm{L},\Delta_\alpha)$ as highest-weight representations. In particular, all the states $\cQ_{\alpha,\nu}$ are linearly independent in this case. Hence $\dim\cV_\alpha^{0,N}=p(N)=\dim \cF_N$ for all $N\in\N$, and we have the linear isomorphism $\cV_\alpha^0\simeq\cF$.
	
	From now on, we assume $\Re(\alpha)<Q$, so that $2Q-\Re(\alpha)>Q$. By \cite[Theorem 1]{Frenkel92_determinant}, the module $\cV_{2Q-\alpha}^0$ is Verma (even if $2Q-\alpha\in kac^+$), namely we have as above $\cV_{2Q-\alpha}^0\simeq M(c_\mathrm{L},\Delta_\alpha)$ as highest-weight representations (recall $\Delta_\alpha=\Delta_{2Q-\alpha}$), and $\cV_{2Q-\alpha}^0=\cF$ as vector spaces. Hence, the canonical projection from the Verma module onto $\cV_\alpha^0$ can be represented as the linear map $\Phi_\alpha^0:\cV_{2Q-\alpha}^0\to\cV_\alpha^0$ such that for all partitions $\nu$,
	\begin{equation}\label{eq:def_phi}
		\Phi_\alpha^0(\bL_{-\nu}^{0,2Q-\alpha}\ind)=\bL_{-\nu}^{0,\alpha}\ind,
	\end{equation}
	i.e. $\Phi_\alpha^0(\cQ_{2Q-\alpha,\nu})=\cQ_{\alpha,\nu}$. The map $\Phi_\alpha^0$ is the canonical linear surjection from the Verma module $\cV_{2Q-\alpha}^0\simeq\cF$ to the highest-weight module $\cV_\alpha^0$. It preserves the level and we denote by $\Phi_\alpha^{0,N}:=\Phi_\alpha^0|_{\cF_N}$ its restriction to level $N$. Moreover, $\ker(\Phi_\alpha^0)$ is a submodule of $\cV_{2Q-\alpha}^0$ and we have a canonical isomorphism of highest-weight Virasoro representations
	\[\cV_\alpha^0=\mathrm{ran}(\Phi_\alpha^0)\simeq\cV_{2Q-\alpha}^0/\ker(\Phi_\alpha^0).\] 
	Again by \cite{Frenkel92_determinant}, the determinant of $\Phi_\alpha^{0,N}$ is given by the celebrated Kac formula:
	\[\det(\Phi_\alpha^{0,N})=\prod_{1\leq rs\leq N}(\Delta_\alpha-\Delta_{\alpha_{r,s}})^{p(N-rs)}.\]
	This implies that $\ker(\Phi_\alpha^0)=\{0\}$ if and only if $\alpha$ does not belong to the Kac table. In particular, $\ker(\Phi_{\alpha_{r,s}}^{0,rs})$ is non-trivial. 
	
	In fact, we claim the stronger statement that $\cV_{\alpha_{r,s}}^0$ is irreducible. This is probably a well-known fact, but we were unable to locate a reference, so we include a short proof. By \cite[Proposition 3.6]{KacRaina_Bombay} (see also Appendix \ref{app:virasoro}), it suffices to show that all singular vectors vanish are multiples of the constant function. Let $\chi\in\cV_{\alpha_{r,s}}^0$ be a singular vector at (strictly) positive level. By definition, we have $\bL_\nu^{0,\alpha_{r,s}}\chi=0$ for all non-empty $\nu\in\cT$. Thus, by the hermiticity relations, $\langle\chi,\cQ_{2Q-\alpha_{r,s},\nu}\rangle_{L^2(\P_{\S^1})}=0$ for all $\nu\in\cT$, i.e. $\chi\in(\cV_{2Q-\alpha_{r,s}}^0)^\perp$. But we have seen that $\cV_{2Q-\alpha_{r,s}}^0=\cF$, so $\chi=0$.
	
	Let us summarise our findings.
	
	\begin{lemma}\label{lem:frenkel}
		Let $c_\mathrm{L}>25$ and $\alpha\in\C$.
		\begin{itemize}
			\item If $\alpha\not\in kac^-$, $\cV_\alpha^0$ is isomorphic to the Verma module $M(c_\mathrm{L},\Delta_\alpha)$.
			\item If $\alpha\in kac^-$, $\cV_\alpha^0$ is isomorphic to the irreducible quotient of the Verma module $\cV_{2Q-\alpha}^0$ by the maximal proper submodule $\ker(\Phi_\alpha^0)$.
		\end{itemize}
	\end{lemma}
	
	 Finally, we show that the map $\alpha\mapsto\Phi_\alpha^{0,N}\in\mathrm{End}(\cF_N)$ is analytic in the region $\Re(\alpha)<Q$ for all levels $N\in\N$. Since the Sugawara expression is polynomial in $\alpha$, there exist polynomial (hence analytic) coefficients $\alpha\mapsto m_{\nu,\mathbf{k}}(\alpha)$ such that
	\[\cQ_{2Q-\alpha,\nu}=\sum_{|\mathbf{k}|=|\nu|}m_{\nu,\mathbf{k}}(2Q-\alpha)\pi_\mathbf{k}.\]
	For $\Re(\alpha)<Q$, $\cV_{2Q-\alpha}^0$ is a Verma module (Lemma \ref{lem:frenkel}), so the states $(\cQ_{2Q-\alpha,\nu})_{|\nu|=N}$ are linearly independent in $\cF_N$ for all $N\in\N$. Since $\dim(\cF_N)=p(N)$, we deduce that these states form a linear basis of $\cF_N$, so that the matrices $m^N(2Q-\alpha):=(m_{\nu,\mathbf{k}}(2Q-\alpha))_{|\nu|=|\mathbf{k}|=N}$ are invertible. Thus, we have an analytic function $\alpha\mapsto m^N(2Q-\alpha)\in\mathrm{GL}_{p(N)}(\C)$ defined in the region $\Re(\alpha)<Q$. Now, the inversion map is analytic on $\mathrm{GL}_{p(N)}(\C)$, so the map $\alpha\mapsto M^N(2Q-\alpha):=m^N(2Q-\alpha)^{-1}$ is well-defined and analytic in the region $\Re(\alpha)<Q$. Hence, for all $\mathbf{k}\in\cT$, we have $\pi_\mathbf{k}=\sum_{|\nu|=|\mathbf{k}|}M_{\mathbf{k},\nu}(2Q-\alpha)\cQ_{2Q-\alpha,\nu}$, for some coefficients $\alpha\mapsto M_{\mathbf{k},\nu}(2Q-\alpha)$ analytic in the region $\Re(\alpha)<Q$. Recalling that $\Phi_\alpha^0(\cQ_{\alpha,\nu})=\cQ_{2Q-\alpha,\nu}$, we have just shown the existence of analytic coefficients $M_{\mathbf{k},\nu}^N(2Q-\alpha)$ in the region $\Re(\alpha)<Q$ such that for all integer partitions $\mathbf{k}$:
	\[\Phi_\alpha^0(\pi_\mathbf{k})=\sum_{|\nu|=|\mathbf{k}|}M_{\mathbf{k},\nu}(2Q-\alpha)\cQ_{\alpha,\nu}.\]
	We stress that the last formula does not extend analytically to the region $\Re(\alpha)>Q$, precisely because some matrices $m^N(\alpha)$ are not invertible on the Kac table, so their inverse $M^N(\alpha)$ will have a pole there.
	
	\begin{example}\label{ex:level_two}
		Frenkel's result implies that $\cV_{2Q-{\alpha_{r,s}}}^0$ is Verma but $\cV_{\alpha_{r,s}}^0$ is not \cite{Frenkel92_determinant}. Thus, the singular vector in $\cV_{2Q-\alpha_{r,s}}^0$ does not vanish. For the benefit of the reader, it may be useful to illustrate this fact on an explicit example at level 2. By a straightforward computation, one finds for all $\alpha\in\C$,
		\begin{align*}
			&\bL_{-2}^{0,\alpha}\ind=2(Q+\alpha)\varphi_2-\varphi_1^2;\qquad(\bL_{-1}^{0,\alpha})^2\ind=2\alpha\varphi_2+\alpha^2\varphi_1^2.
		\end{align*}
		To check when these two polynomials are colinear, we look for linear combinations cancelling $\varphi_1^2$, leading to
		\[(\alpha^2\bL_{-2}^{0,\alpha}+(\bL_{-1}^{0,\alpha})^2)\ind=2\alpha(\alpha^2+\alpha Q+1)\varphi_2=2\alpha(\alpha-\alpha_{1,2})(\alpha-\alpha_{2,1})\varphi_2.\]
		This polynomial vanishes for the degenerate weights $\alpha_{1,2},\alpha_{2,1}<Q$, but not for the dual weights $2Q-\alpha_{1,2},2Q-\alpha_{2,1}$. The root $\alpha=0$ is artificial and corresponds to the level 1 computation: we have $\bL_{-1}^{0,\alpha}=\alpha\varphi_1$, which vanishes for $\alpha=\alpha_{1,1}=0$ (but not for the dual weight $2Q-\alpha_{1,1}=2Q$).
	\end{example}

	\section{Singular modules}\label{sec:poisson}
	In this section, we prove Theorems \ref{thm:poisson} and \ref{thm:singular}, following the outline of Section \ref{subsec:proof_overview}. We introduce and study the singular integrals in Sections \ref{subsec:preliminaries}, \ref{subsec:mero} and \ref{subsec:kac}. Then, we establish the link between singular integrals and descendants states in Section \ref{subsec:descendants}. Finally, Section \ref{subsec:proof_singular} puts the pieces together to conclude the proofs.

	\subsection{Preliminaries}\label{subsec:preliminaries}
	
	Recall $\A_t:=\{e^{-t}<|z|<1\}\subset\C$. For $r\in\N^*$, we set
	\begin{align*}
		&\Delta_t^r:=\{\boldw\in\A_t^r|\,\forall1\leq k<l\leq r,\,|w_k-w_l|>e^{-t}\}\\
		&\Delta_t^{r-1}(w):=\{\boldw\in\A_t^{r-1}|\,(\boldw,w)\in\Delta_t^r\}.
	\end{align*} 
	and for each $\boldw\in\Delta_t^r$,
	\[D_t(\boldw):=\cup_{j=1}^rD(w_j,e^{-t}),\] 
where $D(w,\epsilon)$ is the open disc centred at $w$ and of radius $\epsilon$. Note that $\Delta_t^r$ is invariant under the action of the $r^{\text{th}}$ symmetric group permuting the variables. Given $\boldw=(w_1,...,w_r)\in\A_t^r$ and $\bolds=(s_1,...,s_r)\in\Z^r$, we will use the shorthand $\boldw^{\bolds}:=\prod_{j=1}^rw_j^{s_j}$ throughout the article.

	Let $\alpha<Q$, $t,\epsilon>0$, and $r\in\N$. For $\beta>(Q-\alpha-r\gamma)_+$ arbitrary, we introduce the function $\Psi_\alpha^{t,\epsilon}:\Delta_t^r\to e^{-\beta c}\cD(\cQ)$ defined for $\d c\otimes\P_{\S^1}$-a.e. $(c,\varphi)$ by
	\begin{equation}\label{eq:def_psi_eps}
	\Psi_\alpha^{t,\epsilon}(\boldw):=e^{(\alpha+r\gamma-Q)c}\E_\varphi\left[\epsilon^\frac{\alpha^2}{2}e^{\alpha X_\epsilon(0)}\prod_{j=1}^r\epsilon^\frac{\gamma^2}{2}e^{\gamma X_\epsilon(w_j)}e^{-\mu e^{\gamma c}\int_{\A_t\setminus D_t(\boldw)}\epsilon^\frac{\gamma^2}{2}e^{\gamma X_\epsilon(z)}|\d z|^2}\right].
	\end{equation}
The prefactor $\epsilon^\frac{\alpha^2}{2}e^{\alpha X_\epsilon(0)}\prod_{j=1}^r\epsilon^\frac{\gamma^2}{2}e^{\gamma X_\epsilon(w_j)}$ can be treated thanks to the Cameron--Martin shift, as explained in Section \cite[Section 12.1]{KRV_DOZZ}. Applying this shift also allows us to see that $\Psi_\alpha^{t,\epsilon}(\boldw)$ has a limit in $e^{-\beta c}\cD(\cQ)$ as $\epsilon\to0$, given by
	\begin{equation}\label{eq:def_psi}
		\begin{aligned}
			\Psi_\alpha^t(\boldw)
			&:=\underset{\epsilon\to0}\lim\,\Psi_\alpha^{t,\epsilon}(\boldw)\\
			&=e^{(\alpha+r\gamma-Q)c}\prod_{j=1}^r|w_j|^{-\gamma\alpha}e^{\gamma P\varphi(w_j)-\frac{\gamma^2}{2}G_\del(w_j,w_j)}\prod_{1\leq k<l\leq r}e^{\gamma^2G_\D(w_k,w_l)}\\
			&\qquad\times\E_\varphi\left[e^{-\mu e^{\gamma c}\int_{\A_t\setminus D_t(\boldw)}\prod_{j=1}^re^{\gamma^2G_\D(z,w_j)}\frac{\d M_\gamma(z)}{|z|^{\gamma\alpha}}}\right].
		\end{aligned}
	\end{equation}
	Indeed, the Cameron--Martin shift amounts to a shift of the Dirichlet free field by $\alpha G_\D(0,\cdot\,)+\gamma\sum_{j=1}^rG_\D(w_j,\cdot\,)$. The remaining prefactor in front of the expectation comes from the variance of $\alpha X_\epsilon(0)+\gamma\sum_{j=1}^rX_\epsilon(w_j)$ (with respect to the Dirichlet free field) when applying the Cameron--Martin shift: the divergent part is compensated by the corresponding power of $\epsilon$, while the finite part is a cross-term of Dirichlet Green functions. This has been discussed at length in earlier works, and we refer for instance to the introduction of \cite[Section 3]{DKRV16}.
	
	For each $t>0$ and $\d c\otimes\P_{\S^1}$-a.e., $\Psi_\alpha^t$ is invariant under permutation of its variables. Moreover, $\Psi_\alpha^t(\boldw)\in e^{-\beta c}\cD(\cQ)$ for all $\beta>(Q-\alpha-r\gamma)_+$, and for any $\boldw\in\Delta_t^r$, $\Psi_\alpha^t(\boldw)$ converges as $t\to\infty$ in $e^{-\beta c}\cD(\cQ)$ and $\d c\otimes\P_{\S^1}$-a.e., to the limit 
\begin{equation}\label{eq:expression_psi}	
	\begin{aligned} 
		\Psi_\alpha(\boldw)
		&=e^{(\alpha+r\gamma-Q)c}\prod_{j=1}^r|w_j|^{-\gamma\alpha}e^{\gamma P\varphi(w_j)-\frac{\gamma^2}{2}G_\del(w_j,w_j)}\\
		&\quad\prod_{1\leq k<l\leq r}e^{\gamma^2G_\D(w_k,w_l)}\E_\varphi\left[e^{-\mu e^{\gamma c}\int_\D\prod_{j=1}^re^{\gamma^2G_\D(z,w_j)}\frac{\d M_\gamma(z)}{|z|^{\gamma\alpha}}}\right].
	\end{aligned}
	\end{equation}
	From \cite[appendix C.1]{GKRV20_bootstrap}, we know that $\alpha\mapsto\Psi_\alpha(\boldw)$ admits an analytic continuation in a complex neighbourhood of $(-\infty,Q)$. The states $\Psi_\alpha(\boldw)$ and integrals $\cI_{r,\bolds}(\alpha)$ are always viewed as elements of a weighted space $e^{-\beta c}\cD(\cQ)$ for some arbitrary real number $\beta>(Q-\alpha-r\gamma)_+$.
	
	\begin{remark}
		Given a function $f\in L^1(\Delta_t^r)$, the extension of $f$ by zero in $\D^r\setminus\bar{\Delta}_t^r$ gives a continuous embedding $L^1(\Delta_t^r)\hookrightarrow L^1(\D^r)$, so we will freely identify functions in $L^1(\Delta_t^r)$ with functions in $L^1(\D^r)$.
	\end{remark}

	One key identity to our results is the following derivative formula for $\Psi_\alpha(\cdot)$. This formula has been written in various forms elsewhere \cite{KRV19_local,Oikarinen19_smooth,GKRV20_bootstrap}, and its proof is a consequence of Gaussian integration by parts.
	\begin{lemma}
		For all $F\in\cC$, the map $\boldw\mapsto\E[\Psi_\alpha^t(\boldw)F]$ is smooth over $\Delta_t^r$, and we have
		\begin{equation}\label{eq:derivative_formula}
			\begin{aligned}
				\del_{w_1}\E[\Psi_\alpha^t(\boldw)F]
				&=\left(\alpha\gamma\del_{w_1}G(w_1,0)+\gamma^2\sum_{j=2}^r\del_{w_1}G(w_1,w_j)\right)\E[\Psi_\alpha^t(\boldw)F]\\
				&\quad-\mu\gamma^2\int_{\A_t\setminus D_t(\boldw)}\del_{w_1}G(w_1,w_{r+1})\E[\Psi_\alpha^t(\boldw,w_{r+1})F]|\d w_{r+1}|^2\\
				&\quad+\E[\Psi_\alpha^t(\boldw)\nabla F(\del_{w_1}G_\del(w_1,\cdot\,))]\\
				&\quad -\frac{i}{2}\mu \int_{\partial D(w_1,e^{-t})}\E[\Psi_\alpha^t(\boldw,w_{r+1})F]\d\Bar{w}_{r+1} .
			\end{aligned}
		\end{equation}
		By the penultimate term, we mean the following: for all $w_1\in\D$ and $e^{i\theta}\in\S^1$ we have $\del_{w_1}G_\del(w_1,e^{i\theta})=\frac{1}{2}\frac{e^{-i\theta}}{1-e^{-i\theta}w_1}=\frac{1}{2}\sum_{n=0}^\infty e^{-i(n+1)\theta}w_1^n$, and 
		\begin{equation}\label{eq:expand_nabla}
			\nabla F(\del_{w_1}G_\del(w_1,\cdot\,))=\frac{1}{2}\sum_{n=1}^\infty w_1^{n-1}\del_{-n}F.
		\end{equation}
		For $F\in\cC$, the sum is actually finite (since $F$ depends only on finitely many modes) and $\nabla F(\del_{w_1}G_\del(w_1,\cdot\,))\in\cC$.
	\end{lemma}
	\begin{proof}
The formula is an adaptation of \cite[(3.27)]{KRV19_local} for the differential of Liouville correlation functions, which is based on Gaussian integration by parts. A similar formula appears in \cite[Section 2.3]{Oikarinen19_smooth} (see in particular (2.13) there). Since this is standard material, we will only recall the main steps of the derivation, which is based on Gaussian integration by parts. In principle, one should work with the $\epsilon$-regularisation of the field and pass to the limit; this step is done carefully in \cite[Lemma 3.5]{KRV19_local} and \cite[Section 2.3]{Oikarinen19_smooth}, so we will omit this part. 

 The functional $\Psi_\alpha^t(\boldw)$ depends on $w_1$ through the exponential $e^{\gamma X(w_1)}$, and through the domain of integration of the GMC. Using that $\del_{w_1}e^{\gamma X(w_1)}=\gamma\del_{w_1}X(w_1)e^{\gamma X(w_1)}$, we need to study the insertion of $\del_{w_1}X(w_1)$ inside the expectation. By Gaussian integration by parts, this amounts to differentiating the functional (of the field $X$) inside the expectation in the direction of $\del_{w_1}G(w_1,\cdot\,)$. The functional in question is a product of exponentials of the field at the marked points, and of the Laplace transform of GMC. Applying the Leibniz rule produces a sum of different terms appearing in \eqref{eq:derivative_formula}:
\begin{itemize}
\item The contribution of $\alpha\del_{w_1}G(w_1,0)$ comes from the differential of $e^{\alpha X(0)}$, and that of $\gamma\del_{w_1}G(w_1,w_j)$ comes from the differential of $e^{\gamma X(w_j)}$. The first line of \eqref{eq:derivative_formula} is analogous to the sum in the RHS of \cite[(3.27)]{KRV19_local}.
\item The second line of \eqref{eq:derivative_formula} is the contribution of the differential of the Laplace transform of GMC. It is analogous to the last term in the RHS of \cite[(3.27)]{KRV19_local}.
\item The third line of \eqref{eq:derivative_formula} is the contribution of the differential of $F$ in direction $\del G(w_1,\cdot\,)$. Only $\del G_\del(w_1,\cdot\,)$ contributes since $F$ is independent of $X_\D$. 
\end{itemize}
Finally, the last line of \eqref{eq:derivative_formula} accounts for the variation of the domain of integration of the GMC variable (recall that this domain excludes $D(w_1,e^{-t})$). Indeed, in $\cC'$ we have $\del_{w_1}\int_{D(w_1,e^{-t})}\d M_\gamma(z)=-\frac{i}{2}\int_{\del D(w_1,e^{-t})}e^{\gamma X(w)}\d\bar w$, with the RHS understood in the similar way as the potential $V$ from \eqref{eq:potential}. The precise justification for this identity is the same as the one leading to the expression of $\bL_n$ in \cite[(1.6)]{BGKRV22} in the special case $n=-1$. It is itself a consequence of \cite[Lemma 5.4]{GKRV20_bootstrap} and we refer the reader to its proof for details.
	\end{proof} 
	
	For a partition $(r,\boldsymbol{s})$, we introduce
	\begin{equation}\label{eq:def_I}
		\cI_{r,\boldsymbol{s}}^t(\alpha):=\int_{\Delta_t^r}\Psi_\alpha^t(\boldw)\frac{|\d\boldw|^2}{\boldw^{\bolds}}\in e^{-\beta c}\cD(\cQ).
	\end{equation}
and we recall the notation $\boldw^{\bolds}=\prod_{j=1}^rw_j^{s_j}$.
	
	Our goal is to investigate the existence and analyticity (in $\alpha$) of a limit $\cI_{r,\bolds}$ as $t\to\infty$. The next proposition is a criterion for the absolute integrability of $\Psi_\alpha^t(\cdot)$ on $\D^r$, which refines \cite[Proposition 7.2]{GKRV20_bootstrap}. It gives a bound for arbitrary partitions $(r,\bolds)\in\cT$, but we will actually only need it for the rectangular partition $(r,(s,...,s))$.
	
	\begin{proposition}\label{prop:regular_rectangle}
		Let $(r,\bolds)\in\cT$. There exists a complex neighbourhood of $(-\infty,\alpha_{r,|\bolds|/r})$ such that, $\d c\otimes\P_{\S^1}$-a.e., we have
		\[\underset{t\to\infty}\lim\,\int_{\A_t^r}\Psi_\alpha(\boldw)\frac{|\d\boldw|^2}{|\boldw|^{2\bolds}}<\infty.\]
		As a consequence, for all $\beta>0$, the map $\alpha\mapsto\cI_{r,\bolds}(\alpha)\in e^{-\beta c}\cD(\cQ)$ is analytic in a neighbourhood of $(Q-r\gamma-\beta,\alpha_{r,\frac{|\bolds|}{2r}})$. We denote this neighbourhood by $\cW_{r,\bolds}$ and call it the region of uniform integrability. The same holds if we assume only $(s_1,..,s_r)\in\R_+^r$.
		
	\end{proposition}

For a fixed level $N\in\N$, we can choose $\beta$ large enough so that $\alpha_{r,\frac{\bolds}{2r}}>Q-N\gamma-\beta$ for all $(r,\bolds)\in\cT_N$. For this choice of $\beta$ Proposition \ref{prop:regular_rectangle} gives us a neighbourhood $\cW_{r,\bolds}$ on which the uniform integrability holds. The lower-bound of this neighbourhood is $Q-N\gamma-\beta$ and is uniform over the finite set $\cup_{n=0}^N\cT_n$.
	
	For the empty partition, we will denote the corresponding neighbourhood by $\cW_0$. This neighbourhood contains the interval $(Q-\beta,Q)$ for any $\beta>0$ by \cite[Appendix C.1]{GKRV20_bootstrap}. We also point out that when applied to the rectangular partition $(r,(s,...,s))$, Proposition \ref{prop:regular_rectangle} gives that the map $\alpha\mapsto\cI_{r,(s)}(\alpha)$ is analytic around $\alpha_{r,s}$.
	\begin{proof}[Proof of Proposition \ref{prop:regular_rectangle}]
		Note that the limit exists $\d c\otimes\P_{\S^1}$-a.e. by monotonicity, and it suffices to study the convergence along the sequence of annuli $\A_n$ for integer $n$. For simplicity, we assume $s_1=...=s_r=s$, and we will indicate the changes for the general case at the end of the proof.
		
		Recall the expression for $\Psi_\alpha(\boldw)$ in \eqref{eq:expression_psi}. Integrating over $\P_{\S^1}$ and applying the Cameron--Martin shift to the exponential prefactor in the first line of \eqref{eq:expression_psi}, we have
		\begin{equation}\label{eq:exp_psi}
		\begin{aligned}
			\E[\Psi_\alpha(\boldw)]=e^{(\alpha+r\gamma-Q)c}\prod_{j=1}^r|w_j|^{-\gamma\alpha}&\prod_{1\leq k<l\leq r}|w_k-w_l|^{-\gamma^2}
			\times\E\left[e^{-\mu e^{\gamma c}\int_\D\prod_{j=1}^r|z-w_j|^{-\gamma^2}\frac{\d M_\gamma(z)}{|z|^{\gamma\alpha}}}\right].
		\end{aligned}
		\end{equation}
		Moreover, for all $t>0$, the function $F:\mu\mapsto e^{(\alpha-Q)c}\E\left[e^{-\mu e^{\gamma c}\int_{\A_t}\frac{\d M_\gamma(z)}{|z|^{\gamma\alpha}}}\right]$ is smooth in $(0,\infty)$ (since the GMC is a positive random variable), and we have
		\begin{align*}
			\del_\mu^rF(\mu)
			&=(-\mu)^re^{(\alpha+r\gamma-Q)c}\int_{\A_t^r}\prod_{j=1}^r|w_j|^{-\gamma\alpha}\prod_{1\leq k<l\leq r}|w_k-w_l|^{-\gamma^2}\\
			&\qquad\qquad\qquad\times\E\left[e^{-\mu e^{\gamma c}\int_{\A_t}\prod_{j=1}^r|z-w_j|^{-\gamma^2}\frac{\d M_\gamma(z)}{|z|^{\gamma\alpha}}}\right]|\d\boldw|^2\\
		\end{align*}
		showing that $\boldw\mapsto\Psi_\alpha(\boldw)$ is integrable on $\A_t^r$ for each $t>0$, and 
		\[\int_{\A_t^r}\E[\Psi_\alpha(\boldw)]\frac{|\d\boldw|^2}{|\boldw|^{2s}}\leq e^{2trs}\mu^{-r}|\del_\mu^rF(\mu)|.\]
		
		Now, we prove the result by induction on $r$. For $r=1$, we have $\E[\Psi_\alpha(w)]\leq C|w|^{-\gamma\alpha}$ as $w\to0$, hence $\int_\D\E[\Psi_\alpha(w)]\frac{|\d w|^2}{|w|^{2s}}<\infty$ for $\alpha<\alpha_{1,s}$. In particular, $\int_\D\Psi_\alpha(w)\frac{|\d w|^2}{|w|^{2s}}$ is finite almost surely in this case.
		
		Let $r\geq2$. Suppose the result holds for all $r'<r$, and let us prove it at rank $r$. Given $m<n\in\N$, we use the notation $\A_{m,n}:=\{e^{-n}<|w|<e^{-m}\}$. For all $n\in\N^*$, we write
		\begin{equation}\label{eq:split_annulus}
		\begin{aligned}
			\int_{\A^r_{0,n+1}}\Psi_\alpha(\boldw)\frac{|\d\boldw|^2}{|\boldw|^{2s}}
			=\sum_{r'=0}^{r-1}\binom{r}{r'}\int_{\A_{0,1}^{r-r'}}\int_{\A_{1,n+1}^{r'}}\Psi_\alpha(\boldw)\frac{|\d\boldw|^2}{|\boldw|^{2s}}+\int_{\A_{1,n+1}^r}\Psi_\alpha(\boldw)\frac{|\d\boldw|^2}{|\boldw|^{2s}}.
			\end{aligned}
		\end{equation}
		By the induction hypothesis for $r'$ (since $\frac{|\d\boldw|^2}{|\boldw|^{2s}}$ is bounded on  $\A_{0,1}^{r-r'}$), all the terms in the sum converge $(\d c\otimes\P_{\S^1})$-a.e. as $n\to\infty$. To treat the last term in the RHS, we recall that $X$ is exactly $\log$-correlated in $\D$, i.e. $\E[X(z)X(w)]=\log\frac{1}{|z-w|}$ for all $z,w\in\D$. Hence, by a straightforward computation of covariance, its restriction to $\A_{1,n+1}$ satisfies the identity $X(w)\laweq X(ew)+\delta$, with $\delta$ a standard Gaussian independent of everything. From this identity and \eqref{eq:exp_psi}, we get the following scaling relation: 
		\begin{equation}\label{eq:small_annulus}
			\begin{aligned}
				\int_{\A_{1,n+1}^r}\E[\Psi_\alpha(\boldw)]\frac{|\d\boldw|^2}{|\boldw|^{2s}}
				&=e^{r\gamma\alpha+\frac{1}{2}\gamma^2 r(r-1)+2r(s-1)}\int_{\A_{0,n}^r}\prod_{j=1}^r|w_j|^{-\gamma\alpha}\prod_{1\leq k<l\leq r}|w_k-w_l|^{-\gamma^2}\\
				&\qquad\times e^{(\alpha+r\gamma-Q)c}\E\left[\exp\left(-\mu e^{\gamma c}\int_\D\prod_{j=1}^r|z-\frac{w_j}{e}|^{-\gamma^2}\frac{\d M_\gamma(z)}{|z|^{\gamma\alpha}}\right)\right]\frac{|\d\boldw|^2}{|\boldw|^{2s}}.
			\end{aligned}
		\end{equation}
		Note that the constant prefactor equals $e^{r\gamma(\alpha-\alpha_{r,s})}$. Next, we integrate the last line over the zero mode. We fix $\beta\geq0$ such that $\beta=0$ if $\alpha+r\gamma-Q>0$, and $\beta>Q-\alpha-r\gamma$ otherwise. Writing $\sigma:=\frac{1}{\gamma}(\beta+\alpha+r\gamma-Q)>0$, we have
		\begin{align*}
			&\int_\R e^{(\beta+\alpha+r\gamma-Q)c}\E\left[e^{-\mu e^{\gamma c}\int_\D\prod_{j=1}^r|z-\frac{w_j}{e}|^{-\gamma^2}\frac{\d M_\gamma(z)}{|z|^{\gamma\alpha}}}\right]\d c\\
			&\qquad=\frac{1}{\gamma}\mu^{-\sigma}\Gamma(\sigma)\E\left[\left(\int_\D\prod_{j=1}^r|z-\frac{w_j}{e}|^{-\gamma^2}\frac{\d M_\gamma(z)}{|z|^{\gamma\alpha}}\right)^{-\sigma}\right]\\
			&\qquad\leq\frac{1}{\gamma}\mu^{-\sigma}\Gamma(\sigma)\E\left[\left(\int_{e^{-1}\D}\prod_{j=1}^r|z-\frac{w_j}{e}|^{-\gamma^2}\frac{\d M_\gamma(z)}{|z|^{\gamma\alpha}}\right)^{-\sigma}\right]\\
			&\qquad=\frac{1}{\gamma}\mu^{-\sigma}\Gamma(\sigma)e^{\frac{1}{2}\gamma^2\sigma^2}e^{(-r\gamma^2-\gamma\alpha+\gamma Q)\sigma}\E\left[\left(\int_\D\prod_{j=1}^r|z-w_j|^{-\gamma^2}\frac{\d M_\gamma(z)}{|z|^{\gamma\alpha}}\right)^{-\sigma}\right]\\
			&\qquad=e^{-\frac{1}{2}((\alpha+r\gamma-Q)^2-\beta^2)}\int_\R e^{(\alpha+\beta+r\gamma-Q)c}\E\left[e^{-\mu e^{\gamma c}\int_\D\prod_{j=1}^r|z-w_j|^{-\gamma^2}\frac{\d M_\gamma(z)}{|z|^{\gamma\alpha}}}\right]\d c,
		\end{align*}
		where we used the same scaling relation in the second-to-last line. Plugging this into \eqref{eq:small_annulus}, we get
		\begin{equation}\label{eq:scaling_annulus}
		\begin{aligned}
			&\int_{\A_{1,n+1}^r}\norm{\Psi_\alpha(\boldw)}_{L^1(e^{\beta c}\d c\otimes\P_{\S^1})}\frac{|\d\boldw|^2}{|\boldw|^{2s}}\\
			&\qquad\leq e^{r\gamma(\alpha-\alpha_{r,s})}e^{-\frac{1}{2}((\alpha+r\gamma-Q)^2-\beta^2)}\int_{\A_{0,n}^r}\norm{\Psi_\alpha(\boldw)}_{L^1(e^{\beta c}\d c\otimes\P_{\S^1})}\frac{|\d\boldw|^2}{|\boldw|^{2s}}.
			\end{aligned}
		\end{equation}
		Recall that for each $r'\in\{0,...,r-1\}$, the term $\int_{\A_{0,1}^{r-r'}\times\A_{1,n}^{r'}}\Psi_\alpha(\boldw)\frac{|\d\boldw|^2}{|\boldw|^{2s}}$ converges almost everywhere as $n\to\infty$. By monotone convergence, we deduce that $\int_{\A_{0,1}^{r-r'}\times\A_{1,n}^{r'}}\norm{\Psi_\alpha(\boldw)}_{L^1(e^{\beta c}\d c\otimes\P_{\S^1})}\frac{|\d\boldw|^2}{|\boldw|^{2s}}$ converges to a finite limit, 
		denoted $\int_{A_{0,1}^{r-r'}\times e^{-1}\D^{r'}}\norm{\Psi_\alpha(\boldw)}_{L^1(e^{\beta c}\d c\otimes\P_{S^1})}\frac{|\d\boldw|^2}{|\boldw|^{2s}}$. Going back to \eqref{eq:split_annulus}, we have
		\begin{align*}
			&\int_{\A_{0,n+1}^r}\norm{\Psi_\alpha(\boldw)}_{L^1(e^{\beta c}\d c\otimes\P_{\S^1})}\frac{|\d\boldw|^2}{|\boldw|^{2s}}\\
			&\quad\leq\sum_{r'=0}^{r-1}\binom{r}{r'}\int_{\A_{0,1}^{r-r'}}\int_{e^{-1}\D^{r'}}\norm{\Psi_\alpha(\boldw)}_{L^1(e^{\beta c}\d c\otimes\P_{\S^1})}\frac{|\d\boldw|^2}{|\boldw|^{2s}}\\
			&\qquad+e^{r\gamma(\alpha-\alpha_{r,s})}e^{-\frac{1}{2}((\alpha+r\gamma-Q)^2-\beta^2)}\int_{\A_{0,n}^r}\norm{\Psi_\alpha(\boldw)}_{L^1(e^{\beta c}\d c\otimes\P_{\S^1})}\frac{|\d\boldw|^2}{|\boldw|^{2s}}.
		\end{align*}
		In the case $\alpha+r\gamma-Q>0$, we have $\beta=0$ so $e^{r\gamma(\alpha-\alpha_{r,s})}e^{-\frac{1}{2}(\alpha+r\gamma-Q)^2}<1$ for $\alpha<\alpha_{r,s}$. In the other case, we can choose $\beta$ arbitrarily close to $Q-r\gamma-\alpha$ to make this constant less than 1 as well. It follows that $\underset{n\to\infty}\lim\int_{\A_{0,n}^r}\norm{\Psi_\alpha(\boldw)}_{e^{\beta c}\d c\otimes\P_{\S^1}}\frac{|\d\boldw|^2}{|\boldw|^{2s}}<\infty$, which also implies that the limit $\underset{n\to\infty}\lim\,\int_{\A_{0,n}^r}\Psi_\alpha(\boldw)\frac{|\d\boldw|^2}{|\boldw|^{2s}}<\infty$ holds almost everywhere.
		
		Finally, we explain the changes in the case of arbitrary $\bolds$. First, the binomial coefficient in the RHS of \eqref{eq:split_annulus} must be replaced by a sum over all subsets of $\{1,...,r\}$ of cardinality $r'$. Second, the scaling constant $e^{r\gamma(\alpha-\alpha_{r,s})}$ in \eqref{eq:small_annulus} is replaced by $e^{r\gamma(\alpha-\alpha_{r,|\bolds|/r})}$. The rest of the argument is unchanged.
		
		With this convergence in hand, we can apply dominated convergence to get that $\cI_{r,\bolds}^t\underset{t\to\infty}\to\cI_{r,\bolds}$ locally uniformly in $\cW_{r,\bolds}$. Since $\cI_{r,\bolds}^t$ is analytic for each $t>0$, its local uniform limit is also analytic. 
	\end{proof}

	The next proposition extends the previous one by saying that adding some insertion points with no singularities does not change the integrability properties of the integrals.
	
	\begin{proposition}\label{lem:integrable}
		Let $(r,\bolds)\in\cT$, and $\tilde{r}\in\N^*$. For all $\beta>0$ and $\alpha\in(Q-(r+\tilde{r})\gamma-\beta,\alpha_{r,|\bolds|/r})$, the following limit exists in $e^{-\beta c}\cD(\cQ)$:
		\[\underset{t\to\infty}\lim\,\int_{\A_t^{r+\tilde{r}}}\Psi_\alpha(\boldw,\tilde{\boldw})\frac{|\d\boldw|^2}{|\boldw|^{2\bolds}}|\d\tilde{\boldw}|^2.\]
		As a consequence, the map $\alpha\mapsto\underset{t\to\infty}\lim\,\int_{\A_t^{r+\tilde{r}}}\Psi_\alpha(\boldw,\tilde{\boldw})\frac{|\d\boldw|^2}{\boldw^{\bolds}}|\d\tilde{\boldw}|^2$ is well-defined and analytic on $\cW_{r,\bolds}$, with values in $e^{-\beta c}\cD(\cQ)$.
	\end{proposition}
	\begin{proof}
		In this proof, we write explicitly the dependence on $\mu$ and $\tilde{r}$, i.e. we write $\cI_{r,(s)}^{\mu,\tilde{r}}(\alpha)$, and similarly $\Psi_\alpha^{\mu,\tilde{r}}$.
		
		The function $\mu\mapsto\Psi_\alpha^{\mu,0}(\boldw)$ is analytic in the region $\Re(\mu)>0$, with values in $e^{-\beta c}\cD(\cQ)$. On the other hand, we have
		\begin{equation}\label{eq:kpz}
			\del_\mu^{\tilde{r}}\Psi_\alpha^{\mu,0}(\boldw)=(-\mu)^{\tilde{r}}\int_{\D^{\tilde{r}}}\Psi_\alpha^{\mu,\tilde{r}}(\boldw,\tilde{w}_1,...,\tilde{w}_r)\prod_{j=1}^{\tilde{r}}|\d\tilde{w}_j|^2,
		\end{equation}
		so the RHS lies in $e^{-\beta c}\cD(\cQ)$ since the LHS does. Note that \eqref{eq:kpz} also holds for the regularised function $\Psi_\alpha^{\mu,0,t}(\boldw)$. By uniform absolute convergence for $\alpha\in\cW_{r,s}$ and Cauchy's integral formula on a small (but fixed) loop around $\mu$, we then get for all $\Re(\mu)>0$ and $\alpha\in\cW_{r,s}$:
		\begin{align*}
			\int_{\D^{r+\tilde{r}}}\Psi_\alpha^{\mu,\tilde{r}}(\boldw,\tilde{\boldw})\frac{|\d\boldw|^2}{|\boldw|^{2 s}}|\d\tilde{\boldw}|^2
			&=(-1)^{\tilde{r}}\del_\mu^{\tilde{r}}\int_{\D^r}\Psi_\alpha^{\mu,0}(\boldw)\frac{|\d\boldw|^2}{|\boldw|^{2s}}\\
			&=\frac{(-1)^{\tilde{r}}\tilde{r}!}{2i\pi}\oint\int_{\D^r}\Psi_\alpha^{\tilde{\mu},0}(\boldw)\frac{|\d\boldw|^2}{|\boldw|^{2s}}\frac{\d\tilde{\mu}}{(\tilde{\mu}-\mu)^{\tilde{r}+1}}.
		\end{align*}
		By the previous proposition, the last line is well-defined and analytic with values in $e^{-\beta c}\cD(\cQ)$ for all $\alpha\in\cW_{r,s}$, and it is the limit of its $t$-regularisation.

	\end{proof}
	
	\begin{remark}\label{circle}
  Given a partition $(r,\bolds)$, we will say that an integral is \emph{of type $(r,\bolds)$} if it is of the form 
\begin{equation}\label{eq:type_rs}
\int_{\D^{r+\tilde{r}}}\E[\Psi_\alpha(\boldw,\tilde{\boldw})F]\prod_{j=1}^r\frac{|\d w_j|^2}{w_j^{s_j}}\prod_{k=1}^{\tilde{r}}|\d\tilde{w}_k|^2,
\end{equation}
for some test function $F\in\cC$.  
We will use the same terminology when $\Psi_\alpha$ is replaced by its regularised version $\Psi_\alpha^t$, or if $\Psi_\alpha$ is not paired with a test function (in which case, the integral is an element of $e^{-\beta c}\cD(\cQ)$ upon convergence). By the previous proposition, integrals of type $(r,\bolds)$ converge absolutely in $\cW_{r,\bolds}$ and are analytic there. 
 \end{remark}
	
	Our goal in the remainder of this section is to show that $\cI_{r,\boldsymbol{s}}$ extends analytically to a complex neighbourhood of $(-\infty,Q)$, with values in suitable weighted spaces $e^{-\beta c}\cD(\cQ)$. We will do so in two steps. In Section \ref{subsec:mero}, we prove that $\cI_{r,\boldsymbol{s}}$ has a meromorphic continuation with possible poles on the Kac table. In Section \ref{subsec:kac}, we exclude the possibility that $\cI_{r,\bolds}$ has a pole, i.e. we show that $\cI_{r,\boldsymbol{s}}$ extends analytically on the Kac table.

	\subsection{Meromorphic continuation}\label{subsec:mero}
	In this section, we prove that $\cI_{r,\boldsymbol{s}}$ extends meromorphically to the region $\cW_0$, with possible poles contained in the Kac table. We proceed in two steps. In Proposition \ref{prop:s_is_one}, we prove that the extension holds for the minimal partitions $(N,(1))$, $N\in\N^*$. In Proposition \ref{prop:s_arbitrary}, we prove that the extension holds for all partitions in $\cT_N$, for each $N\in\N^*$. Both results are proved by induction.
	
	Recall the region $\cW_{r,\bolds}\subset\C$ introduced in Proposition \ref{prop:regular_rectangle}. For simplicity, we will write $\cW_r:=\cW_{r,(1)}$, and we note that $\cW_r$ contains a neighbourhood of $(Q-r\gamma-\beta,\alpha_{r,1/2})$ for some arbitrary large $\beta>0$. Since $\alpha_{r+1,1/2}<\alpha_{r,1/2}$, we can (and will) assume that $\cW_{r+1}\subset\cW_r$. We also recall that $\cW_0$ contains a neighbourhood of the half-line $(Q-\beta,Q)$ for $\beta$ arbitrary large.
	\begin{proposition}\label{prop:s_is_one}
		Let $r\in\N$. The function $\cI_{r,(1)}:\cW_r\to\cC'$ extends meromorphically to $\cW_0$. Moreover, the set of its poles is contained in the set $\{\alpha_{r',1}|\,1\leq r'\leq r\}$.
	\end{proposition}
	\begin{proof}
		We will prove that for each $r\in\N^*$, $\cI_{r,(1)}$ can be expressed as a sum of integrals involving only $r-1$ singularities of order 1 (i.e. integrals of the form $\cI_{r-1,(1)}$). The result will then follow by an induction on $r$.
		
		Integrating formula \eqref{eq:derivative_formula} for $\del_{w_r}\Psi_\alpha^t$ on $\Delta^r_t$ with respect to the measure $\prod_{j=1}^{r-1}\frac{|\d w_j|^2}{w_j}|\d w_r|^2$, we get (after an $\epsilon$-regularisation, which is harmless due to the fact that we are on $\Delta^r_t$) for all $F\in\cC$
		\begin{equation}\label{eq:init_rec}
			\begin{aligned}
				&\int_{\Delta^r_t}\E[\Psi_\alpha^t(\boldw)F]\left(\alpha\gamma\del_{w_r}G(w_r,0)+\gamma^2\sum_{j=1}^{r-1}\del_{w_r}G(w_r,w_j)\right)\prod_{j=1}^{r-1}\frac{|\d w_j|^2}{w_j}|\d w_r|^2\\
				&\quad=-\int_{\Delta_t^r}\E\left[\Psi_\alpha^t(\boldw)\nabla F(\del_{w_1}G_\del(w_1,\cdot\,))\right]\prod_{j=1}^{r-1}\frac{|\d w_j|^2}{w_j}|\d w_r|^2\\
				&\qquad+\mu\gamma^2\int_{\Delta_t^{r+1}}\E[\Psi_\alpha^t(\boldw,w_{r+1})F]\prod_{j=1}^{r-1}\frac{|\d w_j|^2}{w_j}\del_{w_r}G(w_r,w_{r+1})|\d w_r|^2|\d w_{r+1}|^2 \\
				&\qquad+\frac{i}{2}\int_{\Delta_t^{r-1}}\int_{\del D_t(w_1,...,w_{r-1})}\E[\Psi_\alpha^t(\boldw)F]\d\bar{w}_r\prod_{j=1}^{r-1}\frac{|\d w_j|^2}{w_j}\\
				&\qquad+\frac{i}{2}\mu\int_{\Delta^r_t}\int_{\partial D(w_r,e^{-t})}\E[\Psi_\alpha^t(\boldw,w_{r+1})F]\d\bar{w}_{r+1}\prod_{j=1}^{r-1}\frac{|\d w_j|^2}{w_j}|\d w_r|^2\\
				&\quad=:\cI_1^t+\cI_2^t+\cI_3^t+\cI_4^t,
			\end{aligned}
		\end{equation}
		where $\cI_1^t$ denotes the first line and so on.

		Let us focus on the singularities coming from the LHS. By permutation symmetry of $\Psi_\alpha^t$ and antisymmetry of $(w_r,w_j)\mapsto\frac{1}{w_r-w_j}$, we have for all $1\leq j\leq r-1$,
		\begin{align*}
			\int_{\Delta_t^2}\Psi_\alpha^t(\boldw)\frac{1}{w_r-w_j}\frac{|\d w_j|^2}{w_j}|\d w_r|^2
			&=\frac{1}{2}\int_{\Delta_t^2}\Psi_\alpha^t(\boldw)\frac{1}{w_r-w_j}\left(\frac{1}{w_j}-\frac{1}{w_r}\right)|\d w_j|^2|\d w_r|^2\\
			&=\frac{1}{2}\int_{\Delta_t^2}\Psi_\alpha^t(\boldw)\frac{|\d w_j|^2}{w_j}\frac{|\d w_r|^2}{w_r}.
		\end{align*}
		Therefore, since $\del_{w_r}G_\D(w_r,0)=-\frac{1}{2w_r}$ and $\del_{w_r}G(w_r,w_1)=-\frac{1}{2}\frac{1}{w_r-w_1}$, we see that the LHS is a multiple of $\cI_{r,(1)}^t$. Namely, \eqref{eq:init_rec} can be rewritten as:
		\begin{equation}\label{eq:show_kac}
			-\frac{\gamma}{2}(\alpha-\alpha_{r,1})\cI_{r,(1)}^t=\cI_1^t+\cI_2^t+\cI_3^t+\cI_4^t,
		\end{equation}
where we recall that $\alpha_{r,1}=\frac{\gamma}{2}(1-r)$.
		
		Let us now study the RHS of \eqref{eq:init_rec}. Recalling \eqref{eq:expand_nabla}, we see that $\cI^t_1$ is treated by the induction hypothesis at rank $r-1$. Using the antisymmetry of $(w_r,w_{r+1})\mapsto\frac{1}{w_r-w_{r+1}}$, we see that $\cI_2^t=0$. Now, we move to $\cI_3^t$. We have $\del D_t(w_1,...,w_{r-1})=\S^1\cup\cC_t(w_1,...,w_{r-1})$, where $\cC_t(w_1,...,w_{r-1})$ is a union of small circles centred at 0, $w_1,...,w_{r-1}$. We write $\cI_3^t=\tilde{\cI}_3^t+\hat{\cI}_3^t$ where $\tilde{\cI}_3^t$ is the contribution of $\S^1$ in $\del D_t(w_1,...,w_{r-1})$, i.e.
  \begin{align*}
			\tilde{\cI}_3^t
			:=\int_{\S^1}\int_{\Delta_t^{r-1}}\E[\Psi_\alpha^t(\boldw)F]\prod_{j=1}^{r-1}\frac{|\d w_j|^2}{w_j}\d\bar{w}_r
		\end{align*}
		By the Girsanov transform, for $w_r\in\S^1$, the change of measure $e^{\gamma\varphi(w_r)-\frac{\gamma^2}{2}G_\del(w_r,w_r)}$ amounts to shifting $\varphi$ by the (deterministic) function $w\mapsto\gamma G_\del(w_r,\cdot)$. Expanding $G_\del(w_r,\cdot)$ in Fourier modes, this means that the modes $\Re(\varphi_n),\Im(\varphi_n)$ are shifted by a deterministic amount, which is $O(1/n)$. Since $F\in\cC$, it depends only on finitely many modes, so the function $\varphi\mapsto F(\varphi+\gamma G_\del(w_r,\cdot))$ also belongs to $\cC$. This implies that $\tilde{\cI}_3^t$ has the same analytic properties as $\int_{\Delta_t^{r-1}}\E[\Psi_\alpha^t(\boldw)F]\prod_{j=1}^{r-1}\frac{|\d w_j|^2}{w_j}\d\bar{w}_r$, i.e. the induction hypothesis at rank $r-1$ applies to $\tilde{\cI}_3^t$. Thus, as a function of $\alpha$, $\tilde{\cI}_3^t$ converges to a limit $\tilde{\cI}_3$ for $\alpha\in\cW_{r-1}$, and $\tilde{\cI}_3$ is analytic in $\cW_{r-1}$. 

  The last remaining terms are $\hat{\cI}_3^t$ and $\cI_4^t$, where we recall that $\hat{\cI}_3^t$ is the contribution of the boundary component $\cC_t(w_1,...,w_{r-1})$ in $\cI_3^t$. Passing to the $\epsilon$-regularisation (i.e. we replace $\Psi_\alpha^t$ with the regularised version $\Psi_\alpha^{t,\epsilon}$ from \eqref{eq:def_psi_eps} in the definition of the integrals), we have that all the other bulk terms in \eqref{eq:init_rec} are uniformly integrable in the region $\cW_{r+1}$. This is then also true for the contribution of $\cC_t(w_1,...,w_{r-1})$. Namely for all $\alpha\in\cW_{r+1}$, writing $\hat{\cI}_3^{t,\epsilon}$ and $\cI_4^{t,\epsilon}$ for the $\epsilon$-regularised integrals,  
  \begin{align*}
\underset{t\to\infty}\lim\,\hat{\cI}_3^t+\cI_4^t=\underset{\epsilon\to0}\lim\,\underset{t\to\infty}\lim\,\hat{\cI}_3^{t,\epsilon}+\cI_4^{t,\epsilon}=0.
\end{align*}
To see that the exchange of limits is licit, observe first that the domain of integration of both $\hat\cI_3^{t,\epsilon}$ and $\cI_4^{t,\epsilon}$ has a total volume of order $e^{-t}$ independently of $\epsilon$ (since it contains a circle of radius $e^{-t}$). Moreover, the integrand is uniformly integrable in the region $\cW_{r+1}$, so the convergence is indeed uniform with respect to $\epsilon$.

		
		Therefore, for all $\alpha\in\cW_{r+1}\subset\cW_{r-1}$, \eqref{eq:show_kac} has a limit as $t\to\infty$, which is given by
		\[-\frac{\gamma}{2}(\alpha-\alpha_{r,1})\cI_{r,(1)}(\alpha)=\cI_1(\alpha)+\tilde{\cI}_3(\alpha),\]
  and the formula extends to $\cW_{r-1}$ by analytic continuation. By the induction hypothesis, this shows that the RHS has a meromorphic continuation to the region $\cW_{0}$, and $\cI_{r,(1)}$ may get an extra pole at $\alpha=\alpha_{1,r}$. This defines a meromorphic extension of $\cI_{r,(1)}$ in the region $\cW_{0}$, with only possible poles at $\{\alpha\mid\alpha=\alpha_{r',1} ,1\leq r'\leq r\}$. Moreover, these poles are at most simple. 
	\end{proof}
	
	\begin{proposition}\label{prop:s_arbitrary}
		Let $(r,\bolds)\in\cT$. The function $\cI_{r,\boldsymbol{s}}:\cW_{r,\boldsymbol{s}}\to\cC'$ extends meromorphically to $\cW_0$. Moreover, the set of its poles is contained in the Kac table.
	\end{proposition}
	\begin{proof}
		
		We will prove the result by induction on $\cT_{|\bolds|}$, i.e. by induction on the distance to the minimal partition as explained in Section \ref{subsec:partitions}. The initialisation holds by the previous proposition since $(|\bolds|,(1))$ is the minimal partition in $\cT_{|\bolds|}$. Now we assume $s_1\geq 2$. For the heredity, we will show that $\cI_{r,\boldsymbol{s}}$ can be expressed as a sum of integrals of strictly lower type. Recall $r^*=\max\{j\mid s_j=s_1\}$ 
		
		By an integration by parts, we get for all $(w_2,...,w_r)\in\Delta_t^{r-1}$
		\begin{align*}
			(s_1-1)\int_{D_t(w_2,...,w_r)}\E[\Psi_\alpha^t(\boldw)F]\frac{|\d w_1|^2}{w_1^{s_1}}
			&=\frac{i}{2}\int_{\del D_t(w_2,...,w_r)}\E[\Psi_\alpha^t(\boldw)F]\frac{\d\bar{w}_1}{w_1^{s_1-1}}\\
			&\quad+\int_{D_t(w_2,...,w_r)}\del_{w_1}\E[\Psi_\alpha^t(\boldw)F]\frac{|\d w_1|^2}{w_1^{s_1-1}}
		\end{align*}
		Combining with \eqref{eq:derivative_formula} integrated with respect to $\frac{|\d w_1|^2}{w_1^{s_1-1}}\prod_{j=2}^r\frac{|\d w_j|^2}{w_j^{s_j}}$, we get
		\begin{equation}\label{eq:rec}
			\begin{aligned}
				&\int_{\Delta_t^r}\left({-}\frac{s_1-1}{w_1}+\alpha\gamma\del_{w_1}G(w_1,0)+\gamma^2\sum_{j=2}^{r^*}\del_{w_1}G(w_1,w_j)\right)\E[\Psi_\alpha^t(\boldw)F]\frac{|\d w_1|^2}{w_1^{s_1-1}}\prod_{k=2}^r\frac{|\d w_k|^2}{w_k^{s_k}}\\
				&\quad=-\gamma^2\sum_{j=r^*+1}^r\int_{\Delta_t^r}\E[\Psi_\alpha^t(\boldw)F]\del_{w_1}G(w_1,w_j)\frac{|\d w_1|^2}{w_1^{s_1-1}}\prod_{k=2}^r\frac{|\d w_k|^2}{w_k^{s_k}}\\
				&\qquad+\int_{\Delta_t^r}\E\left[\Psi_\alpha^t(\boldw)\nabla F(\del_{w_1}G_\del(w_1,\cdot\,))\right]\frac{|\d w_1|^2}{w_1^{s_1-1}}\prod_{k=2}^r\frac{|\d w_k|^2}{w_k^{s_k}}\\
				&\qquad-\mu\gamma^2\int_{\Delta_t^{r+1}}\E[\Psi_\alpha^t(\boldw)F]\del_{w_1}G(w_1,w_{r+1})\frac{|\d w_1|^2}{w_1^{s_1-1}}\prod_{k=2}^r\frac{|\d w_k|^2}{w_k^{s_k}}|\d w_{r+1}|^2\\
				&\qquad-\frac{i}{2}\int_{\del D_t(w_2,...,w_r)}\int_{\Delta_t^{r-1}(w_1)}\E[\Psi_\alpha^t(\boldw)F]\prod_{k=2}^r\frac{|\d w_k|^2}{w_k^{s_k}}\frac{\d\bar{w}_1}{w_1^{s_1-1}}\\
				&\qquad-\frac{i}{2}\mu\int_{\Delta^r_t}\int_{\partial D(w_1,e^{-t})}\E[\Psi_\alpha^t(\boldw,w_{r+1})F]\d\bar{w}_{r+1}\frac{|\d w_1|^2}{w_1^{s_1-1}}\prod_{j=2}^r\frac{|\d w_j|^2}{w_j^{s_j}}\\
				&\quad=:\cJ_1^t+\cJ_2^t+\cJ_3^t+\cJ_4^t+\cJ_5^t.
			\end{aligned}
		\end{equation}
		Let us first focus on the LHS. As in the proof of Proposition \ref{prop:s_is_one}, we can symmetrise the terms involving $\del_{w_1}G(w_1,w_j)=-\frac{1}{2}\frac{1}{w_1-w_j}$ for $1\leq j\leq r^*$ to see that the LHS is a multiple of $\cI_{r,\bolds}^t$. Namely, \eqref{eq:rec} can be written as
		\begin{equation}\label{eq:show_kac_bis}
			-\frac{\gamma}{2}(\alpha-\alpha_{r^*,s_1})\cI^t_{r,\bolds}=\cJ_1^t+\cJ_2^t+\cJ_3^t+\cJ_4^t+\cJ_5^t.
		\end{equation}

		Next, we study the integrals appearing in the RHS of \eqref{eq:rec}, starting with $\cJ_1^t$. If $r^*=r$, the sum defining $\cJ_1^t$ is empty, so we assume $r^*<r$. For $r^*<j\leq r$, by symmetrising $\del_{w_1}G(w_1,w_j)=-\frac{1}{2}\frac{1}{w_1-w_j}$ and using $s_1<s_j$, we get
		\begin{align*}
			&\int_{\Delta_t^2}\Psi_\alpha^t(\boldw)\frac{1}{w_1-w_j}\frac{|\d w_1|^2}{w_1^{s_1-1}}\frac{|\d w_j|^2}{w_j^{s_j}}\\
			&\quad=\frac{1}{2}\int_{\Delta_t^2}\Psi_\alpha^t(\boldw)\frac{1}{w_1-w_j}\left(\frac{1}{w_1^{s_1-1}w_j^{s_j}}-\frac{1}{w_1^{s_j}w_j^{s_1-1}}\right)|\d w_1|^2|\d w_j|^2\\
			&\quad=\frac{1}{2}\int_{\Delta_t^2}\Psi_\alpha^t(\boldw)\frac{1}{w_1^{s_1-1}w_j^{s_1-1}}\frac{w_j^{s_1-s_j-1}-w_1^{s_1-s_j-1}}{w_1-w_j}|\d w_1|^2|\d w_j|^2\\
			&\quad=\frac{1}{2}\sum_{l=1}^{s_1-s_j-1}\int_{\Delta_t^2}\Psi_\alpha^t(\boldw)\frac{|\d w_1|^2}{w_1^{s_1-l}}\frac{|\d w_j|^2}{w_j^{s_j+l}}.
		\end{align*}
 From this formula (and permutation symmetry of the $w_j$'s), it is straightforward to check that all $\cJ_1^t$ is a sum of integrals of type $(r',\bolds')\prec(r,\bolds)$. Therefore, $\cJ_1^t$ is absolutely convergent is a region $\cW_{r',\bolds'}$ for some $(r',\bolds')\prec(r,\bolds)$, and the limit $\cJ_1$ is analytic in this region. 
 
 Recalling \eqref{eq:expand_nabla},  $\cJ_2^t$ is a sum of singular integral of type $\cI_{(r,\mathbf{s}')}$ with $\mathbf{s}'=(s_1-k,s_2,..,s_r)$ for $1\leq k \leq s_1$. We denote its limit by $\cJ_2$, which is absolutely convergent and analytic in the region $\cW_{r',\bolds'}$ for some $(r',\bolds')\prec(r,\bolds)$. 
		
		For $\cJ_3^t$, we use the following symmetrisation and the fact that $s_1\geq2$:
		\begin{align*}
			\int_{\Delta_t^2}\Psi_\alpha^t(\boldw)\frac{1}{w_1-w_{r+1}}\frac{|\d w_1|^2}{w_1^{s_1-1}}|\d w_{r+1}|^2
			&=\frac{1}{2}\int_{\Delta_t^2}\Psi_\alpha^t(\boldw)\frac{1}{w_1-w_j}\left(\frac{1}{w_1^{s_1-1}}-\frac{1}{w_{r+1}^{s_1-1}}\right)|\d w_1|^2|\d w_{r+1}|^2\\
			&=\frac{1}{2}\sum_{l=1}^{s_1-1}\int_{\Delta_t^2}\Psi_\alpha^t(\boldw)\frac{|\d w_1|^2}{w_1^{s_1-l}}\frac{|\d w_{r+1}|^2}{w_{r+1}^{l}}.
		\end{align*}
  Again, this shows that $\cJ_3^t$ converges absolutely to an analytic limit in the region $\cW_{r',\bolds'}$ for some $(r',\bolds')\prec(r,\bolds)$.
		
Next, we treat $\cJ_4^t$, arguing as in the proof of Proposition \ref{prop:s_is_one} for the treatment of $\cI_3^t$. Namely, we write $\cJ_4^t=\tilde{\cJ}_4^t+\hat{\cJ}_4^t$, where $\tilde{\cJ}_4^t$ is the contribution of the unit circle (resp. $\cC_t(w_2,...,w_r)$) coming from $\del D(w_2,...,w_r)$, and $\hat{\cJ}_4^t$ is the contribution of the circles of radius $e^{-t}$ around $0,w_2,...,w_r$. Arguing similarly as in the previous proof, we get that $\tilde{\cJ}_4^t$ converges absolutely to an analytic limit $\tilde{\cJ}_4$ in the region $\cW_{r',\bolds'}$ for some $(r',\bolds')\prec(r,\bolds)$.

Finally, arguing as in the proof of Proposition \ref{prop:s_is_one} (for the treatment of $\hat{\cI}_3^t+\cI_4^t$), we find that $\hat{\cJ}_4^t+\cJ_5^t\to0$ as $t\to\infty$ in the region $\cW_{(r',\bolds')}$ for some $(r',\bolds')\prec(r,\bolds)$. To conclude, we have shown that there is $(r',\bolds')\prec(r,\bolds)$ such that the following convergence holds absolutely in the region $\cW_{r',\bolds'}$.

		
		\begin{equation}\label{eq:extension_mero}
			-\frac{\gamma}{2}(\alpha-\alpha_{r^*,s_1})\cI_{r,\bolds}(\alpha)=\cJ_1(\alpha)+\cJ_2(\alpha)+\cJ_3(\alpha)+\tilde{\cJ}_4(\alpha),
		\end{equation}
and the limit is analytic as a function of $\alpha$. It follows from the induction hypothesis that the RHS extends meromorphically to $\cW_0$, with possible poles on the Kac table. Therefore, $\cI_{r,\bolds}$ also extends meromorphically to $\cW_0$ and has the same poles as the RHS, possibly augmented with $\alpha_{r^*,s_1}$ (which belongs to the Kac table).
	\end{proof}

	%
	%
	
	\subsection{Study on the Kac table}\label{subsec:kac}
	We have shown that $\cI_{r,\boldsymbol{s}}$ extends meromorphically to a complex neighbourhood of $\{\alpha<Q\}$, with possible poles on the Kac table. In this section, we prove some special cases of Theorem \ref{thm:poisson}. Later in Section \ref{subsec:proof_singular}, we will combine these results with representation-theoretic information to deduce Theorems \ref{thm:poisson} \& \ref{thm:singular} in all generality. 

	In the course of our analysis, we will need to treat the case $\gamma^2\not\in\Q$ first. The extension to all $\gamma\in(0,2)$ will be handled by a continuity argument, which is contained in Proposition \ref{prop:gamma_continuous}. Whenever necessary, we will make the $\gamma$-dependence explicit by writing $\cI_{r,\bolds}^\gamma(\alpha)=\cI_{r,\bolds}(\alpha)$.
	
	We will denote $\mathcal{P}_{r, \boldsymbol{s}} \subset k a c^{-}$ as the set of poles of $\mathcal{I}_{r, \boldsymbol{s}}$. It follows from \eqref{eq:extension_mero} that for all $(r, \bolds) \in \mathcal{T}_N$, we have
	$$
	\mathcal{P}_{r, \boldsymbol{s}} \subset\left\{\alpha_{r^*, s_1}\right\} \bigcup\left(\underset{(\widetilde{r}, \tilde{\boldsymbol{s}}) \prec(r, \boldsymbol{s})}{\bigcup} \mathcal{P}_{\widetilde{r}, \widetilde{\boldsymbol{s}}}\right),
	$$
		\begin{proposition}\label{cor:level_rs}
		Suppose $\gamma^2\not\in\Q$, and let $r,s\in\N^*$. For all partitions $(\tilde{r},\tilde{\bolds})\in\cT_{n}$ with $n\leq rs$, the map $\alpha\mapsto\cI_{\tilde{r},\tilde{\bolds}}(\alpha)$ is regular at $\alpha_{r,s}$.
	\end{proposition}

	\begin{proof}
		First, since $\alpha_{r,s}<Q-(\frac{\gamma}{2}r+\frac{2}{\gamma}\frac{s}{2})$, Proposition \ref{prop:regular_rectangle} gives that the integral $\cI_{r,(s)}$ is regular at $\alpha_{r,s}$.
		
		Let $(\tilde{r},\tilde{\bolds})\in\cT_{n}$ arbitrary with $n\leq rs$. By the induction formulae \eqref{eq:init_rec}, \eqref{eq:rec}, we can express $(\alpha-\alpha_{\tilde{r}^*,\tilde{s}_1})\cI_{\tilde{r},\tilde{\bolds}}$ in terms of integrals of lower type. We can iterate the formula until we reach either an integral of type $(r,(s))$ or of empty type, which are both analytic at $\alpha=\alpha_{r,s}$ in virtue of Lemma \ref{lem:integrable} and Proposition \ref{prop:regular_rectangle}. Since $\gamma^2\notin \Q$, the numbers $\{\alpha_{r',s'}|\,r',s'\in\N^*\}$ are pairwise distinct, the possible poles of $\cI_{r,\bolds}$ obtained by this procedure are all distinct from $\alpha_{r,s}$. 
	\end{proof}

	\begin{corollary}\label{prop:regular_descendants}
		Suppose $\gamma^2\not\in\Q$, and let $r,s\in\N^*$. For all partitions $(\tilde{r},\tilde{\bolds})\in\cT_{rs}$ and all $\nu\in\cT$, the map $\alpha\mapsto\bL_{-\nu}\cI_{\tilde{r},\tilde{\bolds}}(\alpha)\in e^{-\beta c}\cD(\cQ)$ is regular at $\alpha=\alpha_{r,s}$.
	\end{corollary}
	\begin{proof}
		We prove this by induction on the length of $\nu$.
		
		The case $\nu=\emptyset$ is the content of Proposition \ref{cor:level_rs}.
		
		Suppose the result holds for all partitions of length $\ell$, for some $\ell\in\N$. Let $\nu_1\geq...\geq\nu_\ell$ be a partition of length $\ell$. Let $n\geq\nu_1$, and let $\tilde{\nu}=(n,\nu_1,...,\nu_\ell)$. By the induction hypothesis, the family $\bL_{-\nu}\cI_{\tilde{r},\tilde{\bolds}}(\alpha)\in e^{-\beta c}\cD(\cQ)$ is analytic in a neighbourhood $\cU$ of $\alpha=\alpha_{r,s}$. Since $\bL_{-n}:e^{-\beta c}\cD(\cQ)\to e^{-\beta c}\cD'(\cQ)$ is bounded, we get that $\bL_{-\tilde{\nu}}\cI_{\tilde{r},\tilde{\bolds}}:\cU\to e^{-\beta c}\cD'(\cQ)$. Moreover, we know that the family $\bL_{-\tilde{\nu}}\cI_{\tilde{r},\tilde{\bolds}}:\cU\setminus\{\alpha_{r,s}\}\to e^{-\beta c}\cD(\cQ)$ is analytic. By Cauchy's integral formula, we then have the following equality in $e^{-\beta c}\cD'(\cQ)$:
		\[\bL_{-\tilde{\nu}}\cI_{\tilde{r},\tilde{\bolds}}(\alpha_{r,s})=\frac{1}{2i\pi}\oint\bL_{-\tilde{\nu}}\cI_{\tilde{r},\tilde{\bolds}}(\alpha)\frac{\d\alpha}{\alpha-\alpha_{r,s}},\]
		where the integral is over a loop around $\alpha_{r,s}$ contained in $\cU$. This shows that $\bL_{-\tilde{\nu}}\cI_{\tilde{r},\tilde{\bolds}}(\alpha_{r,s})$ is actually well-defined in $e^{-\beta c}\cD(\cQ)$, and the map $\alpha\mapsto\cI_{\tilde{r},\tilde{\bolds}}(\alpha)$ is regular at $\alpha_{r,s}$ with values in $e^{-\beta c}\cD(\cQ)$.
	\end{proof}
	
	The final results of this section provide the continuity property that will allow us to extend the results from $\gamma^2\not\in\Q$ to arbitrary $\gamma\in(0,2)$. Essentially, we will prove that the singular integrals $\cI_{r,\bolds}$ are continuous with respect to $\gamma\in(0,2)$, with values in $\cC'$. We record the following elementary for reference purposes.
	
	\begin{lemma}\label{lem:cts}
		Let $F$ be a Fr\'echet space, and $(f_n)_{n\geq0}$ be a bounded sequence in $C^1([0,1];F)$ such that $f_n$ converges pointwise to some function $f:[0,1]\to F$. Then, $f$ is continuous on $[0,1]$.
	\end{lemma}
	
	\begin{proof}
		Suppose first that $(F,\norm{\cdot}_F)$ is Banach. Define the constant $K:=\sup_{n\geq0}\norm{f_n}_{C^1([0,1];F)}$, which is finite by hypothesis. For all $n\geq0$, and all $t_0,t\in[0,1]$, we have 
		\[\norm{f_n(t)-f_n(t_0)}_F\leq\int_{t_0}^t\norm{f_n'(s)}_F\d s\leq K|t-t_0|.\]
 Taking $n\to\infty$ gives $\norm{f(t)-f(t_0)}_F=\underset{n\to\infty}\lim\,\norm{f_n(t)-f_n(t_0)}_F\leq K|t-t_0|$, showing that $f$ is Lipschitz. In the general case, the topology of $F$ is generated by a countable family of seminorms, and we can use the previous estimate for each one of these seminorms.
	\end{proof}
	
Note that the region $\cW_{r,\bolds}$ depends on $\gamma$, and we will denote it by $\cW_{r,\bolds}^\gamma$ if some confusion is possible.	
	
	\begin{proposition}\label{prop:gamma_continuous}
		For all partitions $(r,\bolds)\in\cT$, the function $(\gamma,\alpha)\mapsto\cI_{r,\bolds}^\gamma(\alpha)\in\cC'$ is continuous on the set $\{(\gamma,\alpha)|\,\gamma\in(0,2)\text{ and }\alpha\in\cW_{r,\bolds}^\gamma\}$.
	\end{proposition}
	
	\begin{proof}
Fix $\gamma_0\in(0,2)$ and $\alpha\in\cW_{r,\bolds}^{\gamma_0}$. We will show that $\gamma\mapsto\cI_{r,\bolds}^\gamma(\alpha)$ is continuous in a (real) neighbourhood of $\gamma_0$.
	
Given $\gamma_1,...,\gamma_r$ in a neighbourhood of $\gamma_0$, and $\epsilon>0$, let us consider
		\begin{align*}
			&F^\epsilon_\gamma(\boldw):=e^{(\alpha+\sum_{j=1}^r\gamma_j-Q)c}\E_\varphi\left[\epsilon^\frac{\alpha^2}{2}e^{\alpha X_\epsilon(0)}\prod_{j=1}^r\epsilon^\frac{\gamma_j^2}{2}e^{\gamma_jX_\epsilon(w_j)}e^{-\mu e^{\gamma c}\int_\D\epsilon^\frac{\gamma^2}{2}e^{\gamma X_\epsilon}|\d z|^2}\right];\\
			&\cJ_{r,\bolds}^{\boldsymbol{\gamma},\epsilon}(\gamma):=\int_{\D^r}F_\gamma^\epsilon(\boldw)\frac{|\d\boldw|^2}{\boldw^{\bolds}},
		\end{align*}
		where we write $\boldsymbol{\gamma}=(\gamma_1,...,\gamma_r)$. Note that $\cJ_{r,\bolds}^\gamma=\cI_{r,\bolds}^\gamma(\alpha)$ is our singular integral.
		
		For each $\epsilon>0$, $\cJ_{r,\bolds}^{\boldsymbol{\gamma},\epsilon}$ is differentiable in $\gamma$, and we have in $\cC'$
		\begin{align*}
			\frac{\del\cJ_{r,\bolds}^{\boldsymbol{\gamma},\epsilon}}{\del\gamma}
			&=-\mu\int_\D\E_\varphi\left[\left(c+X_\epsilon(w)+\gamma\log\epsilon\right)\epsilon^\frac{\gamma^2}{2}e^{\gamma X_\epsilon(w)}\epsilon^\frac{\alpha^2}{2}e^{\alpha X_\epsilon(0)}\right.\\
			&\qquad\times\left.\prod_{j=1}^r\epsilon^\frac{\gamma_j^2}{2}e^{\gamma_jX_\epsilon(w_j)}\exp\left(-\mu e^{\gamma c}\int_\D\epsilon^\frac{\gamma^2}{2}e^{\gamma X_\epsilon(z)}|\d z|^2\right)\right]|\d w|^2\frac{|\d\boldw|^2}{\boldw^{\bolds}}\\
			&=-\mu\frac{\del}{\del\gamma_{r+1}}_{|\gamma_{r+1}=\gamma}\cJ_{r+1,(\bolds,0)}^{(\boldsymbol{\gamma},\gamma_{r+1}),\epsilon}
		\end{align*}

Now, we study the limit of this expression as $\epsilon\to0$. Writing $c+X_\epsilon(w)+\gamma\log\epsilon=c+(\bP\varphi)_\epsilon(w)+X_{\D,\epsilon}+\gamma\log\epsilon$, we see that the only possibly problematic term in the limit is $X_{\D,\epsilon}(w)+\gamma\log\epsilon$. This term can be treated by Gaussian integration by parts, similarly to the proof of the derivative formula \eqref{eq:derivative_formula}: the insertion of $X_\D(w)$ in the expectation amounts to differentiating the functional (of $X_\D$) in direction $G_\D(w,\cdot)$. Combining with the Cameron--Martin shift, we get the following limit as $\epsilon\to0$:
\begin{align*}
&\E_\varphi\left[\left(X_{\D,\epsilon}(w)+\gamma\log\epsilon\right)\epsilon^\frac{\gamma^2}{2}e^{\gamma X_{\D,\epsilon}(w)}\epsilon^\frac{\alpha^2}{2}e^{\alpha X_\epsilon(0)}\prod_{j=1}^r\epsilon^\frac{\gamma_j^2}{2}e^{\gamma_jX_{\D,\epsilon}(w_j)}e^{-\mu e^{\gamma c}\int_\D\epsilon^\frac{\gamma^2}{2}e^{\gamma X_\epsilon(z)}|\d z|^2}\right]\\
&\to\left(-\gamma G_\del(w,w)+\alpha G_\D(w,0)+\sum_{j=1}^r\gamma_jG_\D(w,w_j)\right)\E_\varphi\left[e^{-\mu e^{\gamma c}\int_\D e^{\gamma^2G_\D(z,w)}\prod_{j=1}^re^{\gamma\gamma_jG_\D(z,w_j)}\frac{\d M_\gamma(z)}{|z|^{\gamma\alpha}}}\right]\\
&\quad-\mu e^{\gamma c}\int_\D G_\D(w,w')\E_\varphi\left[e^{-\mu e^{\gamma c}\int_\D e^{\gamma^2G_\D(z,w)}e^{\gamma^2G_\D(z,w')}\prod_{j=1}^re^{\gamma\gamma_jG_\D(z,w_j)}\frac{\d M_\gamma(z)}{|z|^{\gamma\alpha}}}\right]|\d w'|^2.
\end{align*}
We have used that the covariance of the regularised field $X_\epsilon$ is the regularised Green's function $G_{\D,\epsilon}$ satisfying $G_{\D,\epsilon}(w,w)=\log\frac{1}{\epsilon}-\gamma G_\del(w,w)+o(1)$ as $\epsilon\to0$.

Recall that $\cW_{r,\bolds}$ is a region of absolute convergence of $\cI_{r,\bolds}$, and $\cW_{r,\bolds}\subset\cW_{r+1,(\bolds,0)}$. Moreover, the presence of the extra logarithmic singularities (coming from Green's function) in the previous display leaves this region unchanged. Altogether, this shows that the function $\gamma\mapsto\frac{\del}{\del\gamma}\cJ^{(\gamma,...,\gamma),\epsilon}_{r,\bolds}(\gamma)$ is uniformly bounded in $\epsilon$, for $\gamma$ in a neighbourhood of $\gamma_0$, i.e. the family $\gamma\mapsto\cI^{\gamma,\epsilon}_{r,\bolds}(\alpha)$ is uniformly bounded in $C^1$. It follows from Lemma \ref{lem:cts} that $\gamma\mapsto\cI_{r,\bolds}^\gamma$ is continuous in a neighbourhood of $\gamma_0$.

	\end{proof}

	\begin{corollary}\label{cor:cts_mero}
		For all partitions $(r,\bolds)\in\cT$, the function $(\gamma,\alpha)\mapsto\cI_{r,\bolds}^\gamma(\alpha)\in\cC'$ is continuous on the set
		\begin{equation}\label{eq:cts_set}
			\{(\gamma,\alpha)|\,\gamma\in(0,2)\text{ and }\alpha\in\cW_0\}\setminus\cup_{r',s'\in\N^*}\left\lbrace(\gamma,(1-r')\frac{\gamma}{2}+(1-s')\frac{2}{\gamma})|\,\gamma\in(0,2)\right\rbrace.
		\end{equation}
		For each $\gamma\in(0,2)$, the function $\alpha\mapsto\cI^\gamma_{r,\bolds}(\alpha)\in\cC'$ is meromorphic in $\cW_0$, with possible poles on the Kac table.
	\end{corollary}
	\begin{proof}
		By the induction formulae \eqref{eq:init_rec}, \eqref{eq:rec}, we can express $(\alpha-\alpha_{\tilde{r}^*,\tilde{s}_1})\cI_{\tilde{r},\tilde{\bolds}}$ in terms of integrals of lower type, with a larger domain of continuity in $(\gamma,\alpha)$ as given by Proposition \ref{prop:gamma_continuous}. Iterating the formula, we find a polynomial $P(\alpha)=\prod_k(\alpha-\alpha_{r_k,s_k})$ such that $P(\alpha)\cI_{\tilde{r},\tilde{\bolds}}^\gamma(\alpha)$ is continuous on $\{(\gamma,\alpha)|\gamma\in(0,2)\text{ and }\alpha\in\cW_0\}$. Thus, $\cI_{\tilde{r},\tilde{\bolds}}$ is continuous on the set \eqref{eq:cts_set}, and the function $\alpha\mapsto\cI_{\tilde{r},\tilde{\bolds}}^\gamma(\alpha)$ is meromorphic in $\cW_0$ for each $\gamma\in(0,2)$, with possible poles in the Kac table.
	\end{proof}

	\subsection{Descendant states and singular integrals}\label{subsec:descendants}
	The goal of this section is to express the descendant states using the singular integrals studied in the previous subsection. The key idea is to extend the approach of \cite[Section 7]{GKRV20_bootstrap}, which constructs the (generalised) eigenstates of the Liouville Hamiltonian by studying the long-time asymptotics of the Liouville semigroup \eqref{eq:liouville_hamiltonian} acting on the free field eigenstates. However, this approach only converges for $\Re(\alpha)$ small (corresponding to our region of absolute convergence $\cW_{r,\bolds}$), and the previous subsections give the necessary material to get a probabilistic expression for the analytic continuation. 
		
	In the statement of the next lemma, $\boldsymbol{\varepsilon}=(\varepsilon_n)_{n\in\N^*}$ is a sequence of complex numbers, and we use the multi-index notation. The differential operator $\frac{\del}{\del\varepsilon_n}$ denotes the complex derivative in $\varepsilon_n$.
	\begin{lemma}\label{lem:expression_descendants}
		Let $F\in\cC$, $\mathbf{k}\in\cT$, and $\alpha\in\C$. For all $t>0$, we have
		\begin{align*}
			e^{t(2\Delta_\alpha+|\mathbf{k}|)}\E\left[e^{-t\bH}\left(\varphi^\mathbf{k}e^{(\alpha-Q)c}\right)F\right]
			&=e^{(\alpha-Q)c}\frac{\del^{\ell(\mathbf{k})}}{\del\boldsymbol{\varepsilon}^\mathbf{k}}_{|\boldsymbol{\varepsilon}=0}\E\left[\exp\left(-\mu e^{\gamma c}\int_{\A_t}\prod_{n\in\N^*}e^{\gamma\Re\left(\frac{\varepsilon_n}{n}z^{-n}\right)}\frac{\d M_\gamma(z)}{|z|^{\gamma\alpha}}\right)\right.\\
			&\qquad\qquad\qquad\qquad\times\left.F\left(\varphi+\sum_{n\in\N^*}\frac{1}{n}\Re(\varepsilon e^{-ni\theta})\right)\right].
		\end{align*}
	\end{lemma}
	
	\begin{proof}
		Recall that we write $X=X_\D+P\varphi$ the GFF in the disc, where $X_\D$ is a Dirichlet GFF and $P\varphi$ is the harmonic extension of the boundary field, which is written in Fourier modes as
		\[P\varphi(z)=2\Re\left(\sum_{n=1}^\infty\varphi_nz^n\right).\]
		Let us define $\varphi_n(t):=\frac{1}{2\pi}\int_0^{2\pi}X(e^{-t+i\theta})e^{-ni\theta}\d\theta$, as well as $X_n(t):=\Re(\varphi_n(t))$, $Y_n(t):=\Im(\varphi_n(t))$. From \cite[Section 4.1]{GKRV20_bootstrap}, $X_n$ and $Y_n$ are independent Ornstein--Uhlenbeck (OU) processes with parameters $(\kappa,\sigma)=(n,\frac{1}{\sqrt{2}})$. See Appendix~\ref{app:ou} for the notation $(\kappa,\sigma)$ of the parameters of this process. 
		
		For notational simplicity, we only treat the case $\varphi^\mathbf{k}=\varphi_n^k$ for some $k\in\N^*$. The general case is then a straightforward consequence of the independence of the processes $(X_n(t),Y_n(t))_{n\in\N^*}$. For each $\varepsilon\in\C$, we consider the martingale
		\begin{align*}
			\cE_\varepsilon(t)
			&:=e^{e^{nt}\left(2\Re(\varepsilon)X_n(t)-\frac{\Re(\varepsilon)^2}{n}\sinh(nt)\right)}e^{e^{nt}\left(-2\Im(\varepsilon)Y_n(t)-\frac{\Im(\varepsilon)^2}{n}\sinh(nt)\right)}\\
			&=e^{e^{nt}\left(2\Re(\varepsilon\varphi_n(t))-\frac{|\varepsilon|^2}{n}\sinh(nt)\right)}
		\end{align*}
		Note that $\cE_\varepsilon(t)$ is real analytic in $\varepsilon$, and 
		\[\cE_\varepsilon(t)=e^{e^{nt}\varepsilon\varphi_n(t)}(1+\bar{\varepsilon}O(1))=(1+\bar{\varepsilon}O(1))\sum_{k=0}^\infty\frac{\varepsilon^k}{k!}e^{nkt}\varphi_n(t)^k,\]
		so that
		\[e^{nkt}\varphi_n(t)^k=\frac{\del^k}{\del\varepsilon^k}_{|\varepsilon=0}\cE_\varepsilon(t).\]
		%
		
		Next, we evaluate the effect of reweighting the measure by the martingale $\cE_\varepsilon(t)$. Since $\Re(\varphi_n(t))$ and $\Im(\varphi_n(t))$ are independent, Appendix \ref{app:ou} tells us that the reweighting amounts to the shifts
\begin{align*}		
		&\Re(\varphi_n(t))\mapsto\Re(\varphi_n(t))+\frac{\Re(\varepsilon)}{n}\sinh(nt);\\
		&\Im(\varphi_n(t))\mapsto\Im(\varphi_n(t))-\frac{\Im(\varepsilon)}{n}\sinh(nt),
		\end{align*}
		or $\varphi_n(t)\mapsto\varphi_n(t)+\frac{\bar{\varepsilon}}{n}\sinh(nt)$. Writing $z=e^{-t+i\theta}\in\D$, the corresponding shift of the field is
		\begin{align*}
			X(z)\mapsto X(e^{-t+i\theta})+\frac{2}{n}\sinh(nt)\Re(\varepsilon e^{-ni\theta})=X(z)+\frac{1}{n}\Re(\varepsilon(z^{-n}-\bar{z}^n)).
		\end{align*}
		By Girsanov's theorem, we then get
		\begin{equation}\label{eq:flow_with_girsanov}
			\begin{aligned}
				e^{(\alpha-Q)c}\E_\varphi\left[\cE_\varepsilon(t)e^{-\mu e^{\gamma c}\int_{\A_t}\frac{\d M_\gamma(z)}{|z|^{\gamma\alpha}}}\right]
				=e^{2\Re(\varepsilon\varphi_n)}e^{(\alpha-Q)c}\E_\varphi\left[\exp\left(-\mu e^{\gamma c}\int_{\A_t}e^{\frac{\gamma}{n}\Re\left(\varepsilon(z^{-n}-\bar{z}^n)\right)}\frac{\d M_\gamma(z)}{|z|^{\gamma\alpha}}\right)\right].
			\end{aligned}
		\end{equation}
		
		Now, we take the expectation over $\varphi$. Again by Girsanov's theorem, reweighting $\P_{\S^1}$ by $e^{2\Re(\varepsilon\varphi_n)-\frac{|\varepsilon|^2}{2n}}$ amounts to the shift $\varphi_n\mapsto\varphi_n+\frac{\bar{\varepsilon}}{2n}$. The corresponding shift of the harmonic extension is $P\varphi(z)\mapsto P\varphi(z)+\frac{1}{n}\Re(\varepsilon\bar{z}^n)$. Thus, integrating \eqref{eq:flow_with_girsanov} over $\P_{\S^1}$ gives
		\begin{equation}
			\begin{aligned}
				&e^{(\alpha-Q)c}\E\left[\cE_\varepsilon(t)e^{-\mu e^{\gamma c}\int_{\A_t}\frac{\d M_\gamma(z)}{|z|^{\gamma\alpha}}}F(\varphi)\right]\\
				&\quad=e^{\frac{|\varepsilon|^2}{2n}}e^{(\alpha-Q)c}\E\left[\exp\left(-\mu e^{\gamma c}\int_{\A_t}e^{\frac{\gamma}{n}\Re(\varepsilon z^{-n})}\frac{\d M_\gamma(z)}{|z|^{\gamma\alpha}}\right)F\left(\varphi+\frac{1}{n}\Re(\varepsilon e^{-ni\theta})\right)\right]
			\end{aligned}
		\end{equation}
		
		Finally, we differentiate $k$ times at $\varepsilon=0$. The RHS coincides with the RHS of the statement of the lemma. As for the LHS, we have
		\begin{align*}
			e^{(\alpha-Q)c}\frac{\del^k}{\del\varepsilon^k}\E\left[\cE_\varepsilon(t)e^{-\mu e^{\gamma c}\int_{\A_t}\frac{\d M_\gamma(z)}{|z|^{\gamma\alpha}}}F(\varphi)\right]
			&=e^{nkt}e^{(\alpha-Q)c}\E\left[\varphi_n(t)^ke^{-\mu e^{\gamma c}\int_{\A_t}\frac{\d M_\gamma(z)}{|z|^{\gamma\alpha}}}F(\varphi)\right]\\
			&=e^{t(2\Delta_\alpha+nk)}\E[e^{-t\bH}\left(\varphi_n^ke^{(\alpha-Q)c}\right)F(\varphi)].
		\end{align*}
		
		As mentioned at the beginning of the proof, the case of general $\mathbf{k}$ follows by reweighting by the product of martingales $\cE_{\varepsilon_n}(t)=e^{e^{nt}(2\Re(\varepsilon_n\varphi_n(t))-\frac{|\varepsilon|^2}{n}\sinh(nt))}$ and using independence.

\end{proof}


The previous lemma will allow us to express descendant states as singular integrals. Let us list two important consequences of this representation. 

\begin{lemma}
Suppose $\gamma^2\not\in\Q$. For all $r,s\in\N^*$ and $\mathbf{k}\in\cT_{rs}$, the map $\alpha\mapsto\cP_\alpha(\varphi^\mathbf{k})\in e^{-\beta c}\cD(\cQ)$ is analytic at $\alpha_{r,s}$.
\end{lemma}

\begin{proof}
Let $F\in\cC$. We will use Lemma \ref{lem:expression_descendants} to evaluate $\E[\cP_\alpha(\varphi^\mathbf{k})F]$ in terms of singular integrals. By the Leibniz rule, the differential operator $\frac{\del^{\ell(\mathbf{k})}}{\del\boldsymbol{\varepsilon}^\mathbf{k}}$ hits alternatively $F(\varphi+\sum_{n\in\N^*}\frac{1}{n}\Re(\varepsilon e^{-ni\theta}))$ and the potential. We have $\del_\varepsilon F(\varphi+\frac{1}{n}\Re(\varepsilon e^{-ni\theta}))=\frac{1}{2n}\del_{-n}F(\varphi+\frac{1}{n}\Re(\varepsilon e^{-ni\theta}))\in\cC$. Taking higher order derivatives, we see that all the terms produced by differentiating $F$ are in $\cC$.

Next, we evaluate the derivative of the potential, starting at order 1:
\begin{align*}
&\frac{\del}{\del\varepsilon_n}e^{-\mu e^{\gamma c}\int_{\A_t}e^{\frac{\gamma}{n}\Re(\varepsilon_nz^{-n})}\d M_\gamma(z)}\\
&\quad=-\mu\frac{\gamma}{2n}e^{\gamma c}\int_{\A_t}e^{\gamma X(w)-\frac{\gamma^2}{2}\E[X(w)^2]}e^{\frac{\gamma}{n}\Re(\varepsilon_nw^{-n})}e^{-\mu e^{\gamma c}\int_{\A_t}e^{\frac{\gamma}{n}\Re(\varepsilon_nz^{-n})}\d M_\gamma(z)}\frac{|\d w|^2}{w^n}.
\end{align*}
When we iterate this formula and evaluate it at $\boldsymbol{\varepsilon}=0$, we see that 
\[\frac{\del^{\ell(\mathbf{k})}}{\del\boldsymbol{\varepsilon}^\mathbf{k}}_{|\boldsymbol{\varepsilon}=0}\E_\varphi\left[e^{\alpha X(0)-\frac{\alpha^2}{2}\E[X(0)^2]}e^{-\mu e^{\gamma c}\int_{\A_t}\prod_{n\in\N^*}e^{\frac{\gamma}{n}\Re(\varepsilon_nz^{-n})}\d M_\gamma(z)}\right]\in\cC'\otimes\C_{rs}[\mu],\]
i.e. it is a polynomial of degree $rs$ in $\mu$ with coefficients in $\cC'$. Moreover, these coefficients are linear combinations of singular integrals at level smaller or equal to $ rs$ studied in the previous subsection. 
By Proposition \ref{cor:level_rs}, all these singular integrals are analytic at $\alpha_{r,s}$, so we deduce that $\alpha\mapsto\cP_\alpha(\varphi^\mathbf{k})$ is analytic at $\alpha_{r,s}$. 

This shows that $\alpha\mapsto\cP_\alpha(\varphi^\mathbf{k})$ is analytic at $\alpha_{r,s}$, with values in $\cC'$. By Cauchy's formula, we have the equality in $\cC'$
\[\cP_{\alpha_{r,s}}(\varphi^\mathbf{k})=\frac{1}{2i\pi}\oint\cP_\alpha(\varphi^\mathbf{k})\frac{\d\alpha}{\alpha-\alpha_{r,s}},\]
where the integral is over a small loop around $\alpha_{r,s}$. On the other hand, as elements of the Verma module $\cV_\alpha$ from \cite[Theorem 4.5]{BGKRV22}, we know that $\alpha\mapsto\cP_\alpha(\varphi^\mathbf{k})$ takes values in $e^{-\beta c}\cD(\cQ)$ away from the Kac table, so the previous display shows that $\cP_{\alpha_{r,s}}(\varphi^\mathbf{k})\in e^{-\beta c}\cD(\cQ)$. The same equation shows that $\alpha\mapsto\cP_\alpha(\varphi^\mathbf{k})$ is analytic around $\alpha_{r,s}$, with values in $e^{-\beta c}\cD(\cQ)$.
%
%
%
\end{proof}

Recall the Virasoro representation $(\bL_n)_{n\in\Z}$ introduced in \cite[Theorem 1.1]{BGKRV22}: for each $n\in\Z$, these operators can be defined as bounded operators $\bL_n:e^{-\beta c}\cD(\cQ)\to e^{-\beta c}\cD'(\cQ)$\footnote{The theorem is stated for $n\geq0$, but it extends easily to $\bL_{-n}=\bL_n^*$ since $|\langle\bL_{-n}F,G\rangle_\cH|=|\langle F,\bL_nG\rangle_\cH|\leq C\norm{F}_{\cD(\cQ)}\norm{G}_{\cD(\cQ)}$ for all $F,G\in\cC$. The same applies to weighted spaces.}. 
\begin{lemma}
Let $r,s\in\N^*$ and $\chi\in\cF$ such that $\alpha\mapsto\cP_\alpha(\chi)\in e^{-\beta c}\cD(\cQ)$ is analytic at $\alpha_{r,s}$. For all $\nu\in\cT$, the map $\alpha\mapsto\bL_{-\nu}\cP_\alpha(\chi)\in e^{-\beta c}\cD(\cQ)$ is analytic at $\alpha_{r,s}$. 
\end{lemma}
\begin{proof}
We work by induction on the length $\ell=\ell(\nu)$. The case $\ell=0$ is true by hypothesis. Suppose the result holds at rank $\ell\geq0$, and let us prove it at rank $\ell+1$. Let $\nu':=(\nu,\nu_{\ell+1})$ be a partition of length $\ell+1$ (with $\nu$ a partition of length $\ell$). Since $\bL_{-\nu_{\ell+1}}:e^{-\beta c}\cD(\cQ)\to e^{-\beta c}\cD'(\cQ)$ is continuous, the map $\alpha\mapsto\bL_{-\nu_{\ell+1}}(\bL_{-\nu}\cP_\alpha(\chi))$ is analytic at $\alpha_{r,s}$, with values in $e^{-\beta c}\cD'(\cQ)$. To show it's indeed valued in $e^{-\beta c}\cD(\cQ)$, we write Cauchy's formula
\[\bL_{-\nu'}\cP_{\alpha_{r,s}}(\chi)=\frac{1}{2i\pi}\oint\bL_{-\nu'}\cP_\alpha(\chi)\frac{\d\alpha}{\alpha-\alpha_{r,s}},\]
where the contour is a small loop around $\alpha_{r,s}$ which doesn't intersect the Kac table, and the equality holds in $e^{-\beta c}\cD'(\cQ)$. As in the previous proof, for $\alpha$ away from the Kac table, the state $\bL_{-\nu'}\cP_\alpha(\chi)$ is an element of the Verma module $\cV_\alpha\subset e^{-\beta c}\cD(\cQ)$ \cite[Theorem 4.5]{BGKRV22} (as a descendant of $\cP_\alpha(\chi)$). Hence the RHS of the previous display defines an element of $e^{-\beta c}\cD(\cQ)$, and so does the LHS.
\end{proof}

\subsection{Conclusion}	\label{subsec:proof_singular}

From the previous subsections, we have a family of Poisson operators:
\begin{equation*}
\left\lbrace\begin{aligned}
	&\cW_0\times\cF&&\to e^{-\beta c}\cD(\cQ)\\
	&(\alpha,\chi)&&\mapsto\cP_\alpha(\chi).
\end{aligned}\right.
\end{equation*}
For each $\chi\in\cF$, the map $\alpha\mapsto\cP_\alpha(\chi)$ is meromorphic in $\cW_0$ with possible poles in the Kac table $kac^-$. For each fixed $\alpha\in\cW_0\setminus kac^-$, $\cP_\alpha:\cF\to e^{-\beta c}\cD(\cQ)$ is a linear map. The explicit expression of $\cP_\alpha(\chi)$ is a linear combination of the singular integrals studied in Section \ref{sec:poisson}.

From the explicit expression $\Phi_\alpha(\bL_{-\nu}^{0,2Q-\alpha}\ind)=\Psi_{\alpha,\nu}$, the composition $\Phi_\alpha=\cP_\alpha\circ\Phi_\alpha^0:\cF\to e^{-\beta c}\cD(\cQ)$ is well-defined and analytic in $\cW_0$. In particular, the poles of $\cP$ may only occur at the zeros of $\Phi_\alpha^0$ (which is consistent with the fact that $\Phi_\alpha^0$ may vanish only on the Kac table). We then define $\cV_\alpha\subset e^{-\beta c}\cD(\cQ)$ by
\[\cV_\alpha:=\mathrm{ran}(\Phi_\alpha)\simeq\cF/\ker(\Phi_\alpha).\]
Moreover, the explicit expression of $\Phi_\alpha$ gives
\[\cV_\alpha=\mathrm{span}\,\{\Psi_{\alpha,\nu}=\bL_{-\nu}\Psi_\alpha|\,\nu\in\cT\}.\]
More precisely, for $\alpha\not\in kac$, \cite[Theorem 4.5]{BGKRV22} showed that $(\cV_\alpha,(\bL_n)_{n\in\Z})$ is a highest-weight representation, and the Poisson operator intertwines the free field and Liouville representations, i.e. $\Phi_\alpha\circ\bL_n^{0,2Q-\alpha}=\bL_n\circ\Phi_\alpha$ (here, we slightly abuse notations by identifying $\bL_n$ with its restriction to $\cV_\alpha$). The Virasoro relations extend to $\cV_{\alpha_{r,s}}$ by continuity in $\alpha$, so that $\cV_{\alpha_{r,s}}$ is also a highest-weight representation, generated by the highest-weight vector $\Psi_{\alpha_{r,s}}$. In particular, $\ker\Phi_{\alpha_{r,s}}$ is a submodule of $(\cV_{2Q-\alpha_{r,s}}^0,(\bL_n^{0,2Q-\alpha_{r,s}})_{n\in\Z})$. It remains to show that it is isomorphic to $\cV_{\alpha_{r,s}}^0$, i.e. $\ker\Phi_{\alpha_{r,s}}=\ker\Phi_{\alpha_{r,s}}^0$. Since $\cV_{\alpha_{r,s}}^0$ is the quotient of the Verma module by the maximal proper submodule, we have immediately $\ker\Phi_{\alpha_{r,s}}\subset\ker\Phi_{\alpha_{r,s}}^0$. The next proposition gives the converse inclusion and ends the proofs of Theorems \ref{thm:poisson} and \ref{thm:singular}.

\begin{proposition}
For all $r,s\in\N^*$, we have $\ker\Phi_{\alpha_{r,s}}^0\subset\ker\Phi_{\alpha_{r,s}}$. Thus, $\alpha\mapsto\cP_\alpha(\chi)$ is analytic at $\alpha_{r,s}$ for all $\chi\in\cF$.
\end{proposition}

\begin{proof}
\emph{Case $\gamma^2\not\in\Q$.}

In this case, $\ker\Phi_{\alpha_{r,s}}^0$ is the Virasoro module (for the representation $(\bL_n^{0,2Q-\alpha_{r,s}})_{n\in\Z}$) generated by the level $rs$-singular vector. Let $\chi_{r,s}\in\cV^0_{2Q-\alpha_{r,s}}\simeq\cF$ be the level-$rs$ singular vector. Then, $\lim\limits_{\alpha\to \alpha_{r,s}}\frac{1}{\alpha-\alpha_{r,s}}\Phi_\alpha^0(\chi_{r,s})=\chi'$ for some $\chi'\in\cF_{rs}$, by analyticity in $\alpha$. From Lemma \ref{cor:level_rs}, $\alpha\mapsto\cP_\alpha(\chi')$ is analytic at $\alpha_{r,s}$, so that $\lim\limits_{\alpha\to \alpha_{r,s}}\frac{1}{\alpha-\alpha_{r,s}}\Phi_\alpha(\chi_{r,s})=\cP_{\alpha_{r,s}}(\chi')$. In particular, $\Phi_{\alpha_{r,s}}(\chi_{r,s})=0$.

It remains to study the descendant states. By the intertwining property and Proposition \ref{prop:regular_descendants}, we have
\begin{align*}
	\frac{1}{\alpha-\alpha_{r,s}}\Phi_{\alpha_{r,s}}(\bL_{-\nu}^{0,2Q-\alpha_{r,s}}\chi_{r,s})=\frac{1}{\alpha-\alpha_{r,s}}\bL_{-\nu}\cP_\alpha(\Phi_\alpha^0(\chi_{r,s}))\underset{\alpha\to\alpha_{r,s}}\to\,\bL_{-\nu}\cP_{\alpha_{r,s}}(\chi').
\end{align*}
In particular, $\Phi_{\alpha_{r,s}}(\bL_{-\nu}^{0,2Q-\alpha_{r,s}}\chi_{r,s})=0$.

\emph{General case.}

Let $(\gamma_n)\in(0,2)^\N$ such that $\gamma^2_n\not\in\Q$ and $\gamma_n\to\gamma$ as $n\to\infty$. Let $\chi\in\cF$. By Corollary \ref{cor:cts_mero}, we have for all $p\in\N^*$,
\[\oint\cP_\alpha^\gamma(\chi)(\alpha-\alpha_{r,s})^{p-1}\d\alpha=\underset{n\to\infty}\lim\,\oint\cP_\alpha^{\gamma_n}(\chi)(\alpha-\alpha_{r,s})^{p-1}\d\alpha=0,\]
where the contour is a small loop around $\alpha_{r,s}$ which is uniformly bounded away from the Kac table. This contour exists since the Kac table is discrete and depends continuously on $\gamma$, and the range of analyticity of the Poisson operator also depends continuously on $\gamma$. This proves that the polar part of $\cP_{\alpha_{r,s}}^\gamma(\chi)$ vanishes at $\alpha_{r,s}$, i.e. $\cP_\alpha^\gamma(\chi)$ is regular at $\alpha_{r,s}$.
\end{proof}

\begin{remark}
For $\gamma^2\not\in\Q$, the previous proof identifies $\ker\Phi_{\alpha_{r,s}}$ with the module generated by $\Phi_{\alpha_{r,s}}(\chi_{r,s}')=\bS_{\alpha_{r,s}}\Psi_{\alpha_{r,s}}'\in e^{-\beta c}\cD(\cQ)$, where we denote $\Psi_\alpha'=\del_\alpha\Psi_\alpha$, and $\bS_{\alpha_{r,s}}$ is the combination of Virasoro generators at level $rs$ creating the singular vector. This vector is of highest-weight $\Delta_{\alpha_{r,s}}+rs$ for the representation $(\bL_n)_{n\in\Z}$, but is a primary of weight $\Delta_{\alpha_{r,s}}$ for the antiholomorphic representation $(\tilde{\bL}_n)_{n\in\Z}$. In particular, $\tilde{\bS}_{\alpha_{r,s}}\bS_{\alpha_{r,s}}\Psi_{\alpha_{r,s}}'$ is a singular vector for the antiholomorphic representation. It turns out that the Poisson operator $\alpha\mapsto\cP_\alpha(\chi_{r,s}')$ has a pole at $\alpha_{r,s}$, which is the phenomenon underlying the higher equations of motions \cite{Zamolodchikov03_HEM}. This has been studied in \cite{BaverezWu_hem} for $rs=2$, and we believe that our method could be extended to arbitrary $(r,s)\in\N^*$.
\end{remark}

\appendix

\section{Martingales of the Ornstein--Uhlenbeck process}\label{app:ou}
Let $(\Omega,(\cF_t)_{t\geq0},\P)$ be a filtered probability space on which is defined a standard Brownian motion $(W_t)_{t\geq0}$. The Ornstein--Uhlenbeck (OU) process with parameters $\kappa,\sigma>0$ is the solution to the SDE
\[\d U_t=-\kappa U_t\d t+\sigma\d W_t.\]
The process $M_t:=e^{\kappa t}U_t$ is a martingale and solves the SDE
\[\d M_t=\sigma e^{\kappa t}\d W_t.\]
Its quadratic variation is given by
\[[M]_t=\sigma^2\int_0^te^{2\kappa s}\d s=\frac{\sigma^2}{\kappa}e^{\kappa t}\sinh(\kappa t).\]
For each $\varepsilon\in\R$, we can form the exponential martingale
\[\cE_\varepsilon(t):=e^{\varepsilon M_t-\frac{\varepsilon^2}{2}[M]_t}.\]
The initial value is $\cE_\varepsilon(0)=e^{\varepsilon U_0}$. This martingale is real analytic in $\varepsilon$, and we can write
\[\cE_\varepsilon(t)=:\sum_{n=0}^\infty M_n(t)\frac{\varepsilon^n}{n!},\]
where each $M_n$ is a martingale. Explicitly, $M_n(t)=[M]_t^{n/2}\mathrm{He}_n(\frac{M_t}{[M]_t^{1/2}})$, where $\mathrm{He}_n$ is the $n^{\text{th}}$ Hermite polynomial. For instance, we have $M_1(t)=M_t$ and $M_2(t)=M_t^2-[M]_t$.

By the Girsanov transform, reweighting the measure by $e^{-\varepsilon U_0}\cE_\varepsilon(t)$ amounts to the shift $M_t\mapsto M_t+\varepsilon[M]_t$, or equivalently $U_t\mapsto U_t+\frac{\varepsilon\sigma^2}{\kappa}\sinh(\kappa t)$. More precisely, we have for every measurable function $F:C^0([0,t_0])\to\R$,
\begin{equation}\label{eq:girsanov_ou}
\E\left[\cE_\varepsilon(t_0)F\left((U_t)_{0\leq t\leq t_0}\right)\right]=e^{\varepsilon U_0}\E\left[F\left(\left(U_t+\frac{\varepsilon\sigma^2}{\kappa}\sinh(\kappa t)\right)_{0\leq t\leq t_0}\right)\right].
\end{equation}
The shifted process and the original one are mutually absolutely continuous up to finite $t_0$, but they are singular for $t_0=\infty$ (since the shifted process diverges to $\mathrm{sgn}(\varepsilon)\infty$ almost surely).

\section{The Virasoro algebra}\label{app:virasoro}
For the reader's convenience, we recall some notions and terminology from the representation theory of the Virasoro algebra. The main reference is \cite{KacRaina_Bombay}. The interested reader will find more information on the structure of degenerate highest-weight modules in \cite{Astashkevich97}, which was pioneered earlier by Feigin--Fuchs \cite{FeiginFuchs}. The textbook \cite{dFMS_BigYellowBook} also gives a thorough review of the material presented in this appendix, see in particular Sections~7.1-2.

The (complex) Witt algebra $\mathfrak{w}$ is the infinite dimensional Lie algebra of holomorphic vector fields in the neighbourhood of the unit circle $\S^1$. The commutation relations read
\[[v\del_z,w\del_z]=(vw'-v'w)\del_z,\qquad\forall\, v\del_z,w\del_z\in\mathfrak{w}.\]
It admits the generators $\bv_n=-z^{n+1}\del_z$, and the commutation relations for the generators is the familiar formula
\[[\bv_n,\bv_m]=(n-m)\bv_{n+m}.\]
Details on the Witt algebra can be found in \cite[Section 1.1]{KacRaina_Bombay}

The \emph{Virasoro cocycle} is the 2-cocycle on $\mathfrak{w}$ given by
\begin{equation}\label{eq:vir_cocycle}
\omega(v\del_z,w\del_z):=\frac{1}{24i\pi}\oint_{\S^1}v'''w\d z.
\end{equation}
In terms of the generators, we have $\omega(\bv_n,\bv_m)=\frac{1}{12}(n^3-n)\delta_{n,-m}$. The \emph{Virasoro algebra} $\mathfrak{v}$ is the corresponding central extension of $\mathfrak{w}$ \cite[Section 1.3]{KacRaina_Bombay}. It is customary to write the generators $(L_n)_{n\in\Z}$ and the central element $\mathbf{c}$, so that the commutation relations read:
\[[L_n,L_m]=(n-m)L_{n+m}+\frac{\mathbf{c}}{12}(n^3-n)\delta_{n,-m};\qquad[L_n,\mathbf{c}]=0.\]
From the commutation relations, we have a Lie subalgebra $\mathfrak{v}^+$ (resp. $\mathfrak{v}^-$) generated by $\{L_n,\,n\geq1\}$ (resp. $\{L_n,\,n\leq-1\}$). An \emph{integer partition} is a non-decreasing sequence of integers with finitely many non-zero terms. The set of integer partitions is denoted $\cT$. The \emph{level} of a partition is the number $|\nu|=\sum_{n=1}^\infty\nu_n$ partitioned by $\nu$. By the Poincar\'e-Birkhoff-Witt theorem, the universal enveloping algebra $\cU(\mathfrak{v}^-)$ admits a linear basis $\{L_{-\nu}=L_{-\nu_1}L_{-\nu_2}\cdots L_{-\nu_\ell}|\,\nu=(\nu_1,...,\nu_\ell)\in\cT\}$. 


A \emph{Virasoro module} is a vector space $V$ together with a linear map $\rho:\mathfrak{v}\to\mathrm{End}(V)$ such that $[\rho(L_n),\rho(L_m)]=(n-m)\rho(L_{n+m})+\frac{\rho(\mathbf{c})}{12}(n^3-n)\delta_{n,-m}$ for all $n,m\in\Z$, and $\rho(\mathbf{c})=\mathrm{c}_\mathrm{L}\mathrm{Id}_V$ for some scalar $c_\mathrm{L}$ called the \emph{central charge}. A \emph{submodule} of $(V,\rho)$ is a linear subspace $V'\subset V$ such that $(V',\rho_{|V'})$ is again a Virasoro module. A module $(V,\rho)$ is \emph{irreducible} if it has no non-trivial submodule. 

Given a module $(V,\rho)$ as above and $v\in V$, we will write $L_n\cdot v:=\rho(L_n)v$ for simplicity. A vector $v\in V$ is a \emph{highest-weight vector} of (highest) weight $\Delta\in\C$ if $L_n\cdot v=0$ for all $n\geq1$ and $L_0\cdot v=\Delta v$. A module $V$ is a \emph{highest-weight module} of (highest) weight $\Delta$ if $V$ is spanned (algebraically) by $\{L_{-\nu}\cdot v|\,\nu\in\cT\}$ for some highest-weight vector $v$ \cite[Section 3.2]{KacRaina_Bombay}. Namely, $V$ is obtained by acting on $v$ with $\cU(\mathfrak{v}^-)$. From the commutation relations, $L_{-\nu}\cdot v$ is an eigenstate of $L_0$ with eigenvalue $\Delta+|\nu|$. This gives a grading of $V$ by the level: $V=\oplus_{N\in\N}V_N$. The space $V_0$ is one-dimensional (spanned by $v$), and the vectors at level $N>0$ are the \emph{descendants} of $v$. It is customary to parametrise the central charge and the highest weight with the so-called \emph{Liouville parameters} $\gamma,\alpha\in\C$:
\[c_\mathrm{L}=1+6Q^2;\qquad\Delta=\frac{\alpha}{2}(Q-\frac{\alpha}{2}),\]
and $Q=\frac{\gamma}{2}+\frac{2}{\gamma}$. Notice that the parametrisation $(\gamma,\alpha)\mapsto(c_\mathrm{L},\Delta)$ is not one-to-one, since $\Delta$ is invariant under $\alpha\mapsto 2Q-\alpha$ and $c_\mathrm{L}$ is invariant under $\gamma\mapsto\frac{4}{\gamma}$.

A \emph{Verma module} is a highest-weight module such that the canonical map $\cU(\mathfrak{v}^-)\to V,\,L_{-\nu}\mapsto L_{-\nu}\cdot v$ is a linear isomorphism: in other words, the states $\{L_{-\nu}\cdot v|\,\nu\in\cT\}$ are linearly independent \cite[Definition 3.2]{KacRaina_Bombay}. For each $c_\mathrm{L},\Delta\in\C$, the Verma module with central charge $c_\mathrm{L}$ and highest-weight $\Delta$ exists and is unique up to isomorphism \cite[Section 3.3]{KacRaina_Bombay}. It is denoted $M(c_\mathrm{L},\Delta)$. Any highest-weight module $V$ is isomorphic to the quotient of $M(c_\mathrm{L},\Delta)$ by a proper submodule. We then have a canonical projection $M(c_\mathrm{L},\Delta)\to V$ sending $L_{-\nu}$ to the descendant state $L_{-\nu}\cdot v\in V$. A vector $w$ in the highest-weight representation $V$ is \emph{singular} if it is annihilated by all $L_n$ for $n\geq1$, so that it generates its own highest-weight module. In particular, if $M(c_\mathrm{L},\Delta)$ contains a singular vector at positive level, then it is reducible. The converse is also true: by \cite[Proposition 3.6]{KacRaina_Bombay}, $M(c_\mathrm{L},\Delta)$ is irreducible if and only if the only singular vectors are multiples of the highest-weight vector. By \cite[Proposition 3.3]{KacRaina_Bombay}, $M(c_\mathrm{L},\Delta)$ has a unique maximal proper submodule $I(c_\mathrm{L},\Delta)$ containing all other proper submodules. The quotient 
\[V(c_\mathrm{L},\Delta):=M(c_\mathrm{L},\Delta)/I(c_\mathrm{L},\Delta)\] 
is irreducible. Note that $I(c_\mathrm{L},\Delta)=\{0\}$ if and only if $M(c_\mathrm{L},\Delta)$ is irreducible.

For each $\nu,\nu'\in\cT$ with the same level $N$, $L_{\nu'}\cdot (L_{-\nu}\cdot v)$ is at level zero, so we have $L_{\nu'}\cdot(L_{-\nu}\cdot v)=S_N(\nu,\nu')v$ for some scalar $S_N(\nu,\nu')$. The matrix $(S_N(\nu,\nu'))_{|\nu|=|\nu'|=N}$ is called the \emph{Shapovalov matrix} and it defines a Hermitian form on the level-$N$ subspace of $M(c_\mathrm{L},\Delta)$. The collection of all these finite-dimensional forms defines a Hermitian form on $M(c_\mathrm{L},\Delta)$, called the Shapovalov form. By construction, the Shapovalov-adjoint of $L_n$ is $L_{-n}$ for all $n\in\Z$, and the Shapovalov form is characterised by this property, modulo multiplicative constant. It is non-degenerate for all $N\in\N$ if and only if $M(c_\mathrm{L},\Delta)$ is irreducible. More precisely, $I(c_\mathrm{L},\Delta)$ coincides with the kernel of the Shapovalov form. See \cite[Proposition 3.4]{KacRaina_Bombay} for details. The Kac determinant formula gives an explicit expression for the determinant of the Shapovalov matrix \cite[Theorem 8.1]{KacRaina_Bombay}:
\[\det S_N\propto\prod_{1\leq rs\leq N}(\Delta-\Delta_{r,s})^{p(N-rs)},\]
where $\Delta_{r,s}$ is the weight corresponding to $\alpha_{r,s}=Q-\frac{\gamma}{2}r-\frac{2}{\gamma}s$, with $r,s\in\N^*$. In particular, the Shapovalov form is degenerate at some level if and only if $\alpha$ belongs to the \emph{Kac table} $kac=kac^-\sqcup kac^+$, with $kac^\pm=(1\pm\N^*)\frac{\gamma}{2}+(1\pm\N^*)\frac{2}{\gamma}$. The two sets satisfy $kac^-=2Q-kac^+$, and $kac^-\subset\R_-$.

\bibliographystyle{alpha}
\bibliography{bpz}

\newcommand{\etalchar}[1]{$^{#1}$}
\begin{thebibliography}{GSLH{\etalchar{+}}20}

\bibitem[ACSW21]{ACSW2}
Morris Ang, Gefei Cai, Xin Sun, and Baojun Wu.
\newblock Integrability of {Conformal} {Loop} {Ensemble}: {Imaginary} {DOZZ}
  {Formula} and {Beyond}.
\newblock Preprint, {arXiv}:2107.01788 [math-ph] (2021), 2021.

\bibitem[ACSW24]{ACSW1}
Morris Ang, Gefei Cai, Xin Sun, and Baojun Wu.
\newblock {SLE} {Loop} {Measure} and {Liouville} {Quantum} {Gravity}.
\newblock Preprint, {arXiv}:2409.16547 [math.{PR}] (2024), 2024.

\bibitem[{Ang}25]{Ang23_zipper}
Morris {Ang}.
\newblock {Liouville conformal field theory and the quantum zipper}.
\newblock {\em Ann. Probab. (to appear)}, 2025.

\bibitem[ARS25]{ARS25_FZZ}
Morris Ang, Guillaume Remy, and Xin Sun.
\newblock {FZZ} formula of boundary {Liouville} {CFT} via conformal welding.
\newblock {\em J. Eur. Math. Soc. (JEMS)}, 27(3):1209--1266, 2025.

\bibitem[ARSZ23]{ARSZ23}
Morris Ang, Guillaume Remy, Xin Sun, and Tunan Zhu.
\newblock Derivation of all structure constants for boundary {Liouville} {CFT}.
\newblock Preprint, {arXiv}:2305.18266 [math.{PR}] (2023), 2023.

\bibitem[Ast97]{Astashkevich97}
A.~Astashkevich.
\newblock On the structure of {V}erma modules over {V}irasoro and
  {N}eveu-{S}chwarz algebras.
\newblock {\em Comm. Math. Phys.}, 186(3):531--562, 1997.

\bibitem[BB10]{Belavin2_bHEM}
A.~Belavin and V.~Belavin.
\newblock Higher equations of motion in boundary {L}iouville field theory.
\newblock {\em J. High Energy Phys.}, (2):010, 18, 2010.

\bibitem[Ber17]{Berestycki17}
Nathana\"{e}l Berestycki.
\newblock An elementary approach to {G}aussian multiplicative chaos.
\newblock {\em Electron. Commun. Probab.}, 22:Paper No. 27, 12, 2017.

\bibitem[BGK{\etalchar{+}}24]{BGKRV22}
Guillaume Baverez, Colin Guillarmou, Antti Kupiainen, R\'emi Rhodes, and
  Vincent Vargas.
\newblock The {V}irasoro structure and the scattering matrix for {L}iouville
  conformal field theory.
\newblock {\em Probab. Math. Phys.}, 5(2):269--320, 2024.

\bibitem[BJ24]{BJ24}
Guillaume Baverez and Antoine Jego.
\newblock The {CFT} of {SLE} loop measures and the {Kontsevich}--{Suhov}
  conjecture.
\newblock Preprint, {arXiv}:2407.09080 [math.{PR}] (2024), 2024.

\bibitem[BPZ84]{BPZ84}
A.A. Belavin, A.M. Polyakov, and A.B. Zamolodchikov.
\newblock Infinite conformal symmetry in two-dimensional quantum field theory.
\newblock {\em Nuclear Physics B}, 241(2):333--380, 1984.

\bibitem[BW23]{BaverezWu_hem}
Guillaume {Baverez} and Baojun {Wu}.
\newblock {Higher equations of motion at level 2 in Liouville CFT}.
\newblock {\em arXiv e-prints}, page arXiv:2312.13900, December 2023.

\bibitem[Car92]{Cardy_crossing}
John~L. Cardy.
\newblock Critical percolation in finite geometries.
\newblock {\em J. Phys. A, Math. Gen.}, 25(4):l201--l206, 1992.

\bibitem[{Cer}25]{Cercle_HEM}
Baptiste {Cercl\'e}.
\newblock {Higher equations of motion for boundary Liouville Conformal Field
  Theory from the Ward identities}.
\newblock {\em Comm. Math. Phys. (to appear)}, 2025.

\bibitem[CH25]{CercleHuguenin}
Baptiste Cercl{\'e} and Nathan Huguenin.
\newblock Higher-spin symmetry in the $\mathfrak{sl}_3$ boundary {Toda}
  conformal field theory {II}: {Singular} vectors and {BPZ} equations.
\newblock Preprint, {arXiv}:2503.20548 [math.{PR}] (2025), 2025.

\bibitem[DCS12]{DuminilSmirnov}
Hugo Duminil-Copin and Stanislav Smirnov.
\newblock The connective constant of the honeycomb lattice equals
  {{\(\sqrt{2+\sqrt 2}\)}}.
\newblock {\em Ann. Math. (2)}, 175(3):1653--1665, 2012.

\bibitem[DFMS97]{dFMS_BigYellowBook}
Philippe Di~Francesco, Pierre Mathieu, and David S\'{e}n\'{e}chal.
\newblock {\em Conformal field theory}.
\newblock Graduate Texts in Contemporary Physics. Springer-Verlag, New York,
  1997.

\bibitem[DKRV16]{DKRV16}
Fran\c{c}ois David, Antti Kupiainen, R\'{e}mi Rhodes, and Vincent Vargas.
\newblock Liouville quantum gravity on the {R}iemann sphere.
\newblock {\em Comm. Math. Phys.}, 342(3):869--907, 2016.

\bibitem[DMS21]{MatingOfTrees}
Bertrand Duplantier, Jason Miller, and Scott Sheffield.
\newblock Liouville quantum gravity as a mating of trees.
\newblock {\em Ast\'{e}risque}, (427):viii+257, 2021.

\bibitem[DO94]{DornOtto94}
H.~Dorn and H.-J. Otto.
\newblock Two- and three-point functions in {L}iouville theory.
\newblock {\em Nuclear Phys. B}, 429(2):375--388, 1994.

\bibitem[FF90]{FeiginFuchs}
B.~L. Fe\u{\i}gin and D.~B. Fuchs.
\newblock Representations of the {V}irasoro algebra.
\newblock In {\em Representation of {L}ie groups and related topics}, volume~7
  of {\em Adv. Stud. Contemp. Math.}, pages 465--554. Gordon and Breach, New
  York, 1990.

\bibitem[Fre92]{Frenkel92_determinant}
Edward Frenkel.
\newblock Determinant formulas for the free field representations of the
  {V}irasoro and {K}ac-{M}oody algebras.
\newblock {\em Phys. Lett. B}, 286(1-2):71--77, 1992.

\bibitem[GKRV24]{GKRV20_bootstrap}
Colin Guillarmou, Antti Kupiainen, R{\'e}mi Rhodes, and Vincent Vargas.
\newblock Conformal bootstrap in {Liouville} theory.
\newblock {\em Acta Math.}, 233(1):33--194, 2024.

\bibitem[GQW25]{GQW25}
Maria Gordina, Wei Qian, and Yilin Wang.
\newblock Infinitesimal conformal restriction and unitarizing measures for
  {Virasoro} algebra.
\newblock {\em J. Math. Pures Appl. (9)}, 195:24, 2025.
\newblock Id/No 103669.

\bibitem[GSJS21]{JacobsenSaleur21}
Linnea Grans-Samuelsson, Jesper~Lykke Jacobsen, and Hubert Saleur.
\newblock {The action of the Virasoro algebra in quantum spin chains. Part I.
  The non-rational case}.
\newblock {\em JHEP}, 02:130, 2021.

\bibitem[GSLH{\etalchar{+}}20]{JacobsenSaleur}
Linnea Grans-Samuelsson, Lawrence Liu, Yifei He, Jesper~Lykke Jacobsen, and
  Hubert Saleur.
\newblock {The action of the Virasoro algebra in the two-dimensional Potts and
  loop models at generic $Q$}.
\newblock {\em JHEP}, 10:109, 2020.

\bibitem[Kah85]{Kahane85}
Jean-Pierre Kahane.
\newblock Sur le chaos multiplicatif.
\newblock {\em Ann. Sci. Math. Qu\'{e}bec}, 9(2):105--150, 1985.

\bibitem[KRR13]{KacRaina_Bombay}
Victor~G Kac, Ashok~K Raina, and Natasha Rozhkovskaya.
\newblock {\em Bombay Lectures on Highest Weight Representations of Infinite
  Dimensional Lie Algebras}.
\newblock WORLD SCIENTIFIC, 2nd edition, 2013.

\bibitem[KRV19]{KRV19_local}
Antti Kupiainen, R\'{e}mi Rhodes, and Vincent Vargas.
\newblock Local conformal structure of {L}iouville quantum gravity.
\newblock {\em Comm. Math. Phys.}, 371(3):1005--1069, 2019.

\bibitem[KRV20]{KRV_DOZZ}
Antti Kupiainen, R\'{e}mi Rhodes, and Vincent Vargas.
\newblock Integrability of {L}iouville theory: proof of the {DOZZ} formula.
\newblock {\em Ann. of Math. (2)}, 191(1):81--166, 2020.

\bibitem[Nie82]{Nienhuis1}
Bernard Nienhuis.
\newblock {Exact critical point and critical exponents of $O(n)$ models in
  two-dimensions}.
\newblock {\em Phys. Rev. Lett.}, 49:1062, 1982.

\bibitem[Nie84]{Nienhuis2}
Bernard Nienhuis.
\newblock Critical behavior of two-dimensional spin models and charge asymmetry
  in the coulomb gas.
\newblock {\em Journal of Statistical Physics}, 34:731--761, 1984.

\bibitem[Oik19]{Oikarinen19_smooth}
Joona Oikarinen.
\newblock Smoothness of correlation functions in {L}iouville conformal field
  theory.
\newblock {\em Ann. Henri Poincar\'{e}}, 20(7):2377--2406, 2019.

\bibitem[Pel19]{Peltola_sleReview}
Eveliina Peltola.
\newblock Toward a conformal field theory for {Schramm}-{Loewner} evolutions.
\newblock {\em J. Math. Phys.}, 60(10):103305, 39, 2019.

\bibitem[Pol70]{Polyakov70}
Alexander~M. Polyakov.
\newblock {Conformal symmetry of critical fluctuations}.
\newblock {\em JETP Lett.}, 12:381--383, 1970.

\bibitem[PS22]{SantachiaraPicco}
Marco Picco and Raoul Santachiara.
\newblock On the {CFT} describing the spin clusters in 2d {P}otts model.
\newblock {\em Journal of Statistical Mechanics: Theory and Experiment},
  2022(2):023102, feb 2022.

\bibitem[RS15]{RibaultSantachiara}
Sylvain Ribault and Raoul Santachiara.
\newblock Liouville theory with a central charge less than one.
\newblock {\em J. High Energy Phys.}, 2015(8):26, 2015.
\newblock Id/No 109.

\bibitem[RV14]{rhodes2014_gmcReview}
R\'{e}mi Rhodes and Vincent Vargas.
\newblock Gaussian multiplicative chaos and applications: a review.
\newblock {\em Probab. Surv.}, 11:315--392, 2014.

\bibitem[Smi01]{Smirnov_crossing}
Stanislav Smirnov.
\newblock Critical percolation in the plane: {Conformal} invariance, {Cardy}'s
  formula, scaling limits.
\newblock {\em C. R. Acad. Sci., Paris, S{\'e}r. I, Math.}, 333(3):239--244,
  2001.

\bibitem[Tes95]{Teschner_DOZZ}
Jorg Teschner.
\newblock {On the Liouville three point function}.
\newblock {\em Phys. Lett. B}, 363:65--70, 1995.

\bibitem[Zam04]{Zamolodchikov03_HEM}
A.~Zamolodchikov.
\newblock {Higher equations of motion in Liouville field theory}.
\newblock {\em Int. J. Mod. Phys. A}, 19S2:510--523, 2004.

\bibitem[Zha21]{Zhan21}
Dapeng Zhan.
\newblock {SLE} loop measures.
\newblock {\em Probab. Theory Relat. Fields}, 179(1-2):345--406, 2021.

\bibitem[{Zhu}20]{Zhu20_BPZ}
Tunan {Zhu}.
\newblock {Higher order BPZ equations for Liouville conformal field theory}.
\newblock {\em arXiv e-prints}, page
  \href{https://arxiv.org/abs/2001.08476}{arXiv:2001.08476}, January 2020.

\bibitem[ZZ96]{ZZ95}
Alexander~B. Zamolodchikov and Alexei~B. Zamolodchikov.
\newblock {Structure constants and conformal bootstrap in Liouville field
  theory}.
\newblock {\em Nucl. Phys. B}, 477:577--605, 1996.

\end{thebibliography}
\end{document}